\documentclass[leqno,11pt]{article}

\usepackage{booktabs}
\usepackage{color}
\usepackage{amsthm}
\usepackage{amsmath}
\usepackage{amssymb}
\usepackage{graphicx}
\usepackage{listings}
\usepackage[symbol]{footmisc}
\usepackage{mathrsfs}
\usepackage{MnSymbol, graphicx}
\usepackage{mathtools}
\usepackage{url}
\usepackage{datetime}
\usepackage{algorithm}
\usepackage{algorithmicx}
\usepackage{forest}
\usepackage{tikz}
\usepackage{scrextend}
\usepackage{dsfont}
\usepackage{xcolor}
\usepackage{varwidth}
\usepackage{fancyhdr}
\usepackage{yfonts}
\usepackage[english]{babel}
\usepackage{caption}
\usepackage{subcaption}
\usepackage{orcidlink}
\usepackage{authblk}
\usepackage{hyperref}
\usepackage{cite}

\usepackage{booktabs}
\usepackage{siunitx}
\usepackage{orcidlink}

\usepackage{algpseudocode}
\usepackage{enumitem}
\newcommand{\numberset}{\mathbb}
\newcommand{\N}{\numberset{N}}

\newcommand{\R}{\numberset{R}}

\newcommand{\I}{\numberset{I}}

\newcommand{\eqcolon}{\mathrel{\rotatebox[units=360]{180}{\ensuremath{\coloneq}}}}

\DeclareRobustCommand{\rchi}{{\mathpalette\irchi\relax}}
\newcommand{\irchi}[2]{\raisebox{\depth}{$#1\chi$}}

\theoremstyle{plain} 
\newtheorem{thm}{Theorem}[section]      
\newtheorem{cor}[thm]{Corollary} 
 
\newtheorem*{lem*}{Lemma}
\newtheorem{prop}[thm]{Proposition} 
\newtheorem*{theorem*}{Theorem}

\theoremstyle{plain} 
\newtheorem{defn}[thm]{Definition}
\newtheorem*{defn*}{Definition}
\newtheorem{ex}[thm]{Example}
\newtheorem*{ex*}{Example}

\theoremstyle{plain} 
\newtheorem{obs}[thm]{Remark} 
\newtheorem*{obs*}{Remark} 
 
\newtheorem*{prob*}{Problem} 

\algblock[Step]{Step}{EndStep}

\title{New solutions for the symmetrical n-body problem through variational approach and optimisation techniques}

\author{Roberto Ciccarelli, Margaux Introna, \\ Susanna Terracini, Massimiliano Vasile}

\date{}

\begin{document}

\maketitle

\abstract{Advances in the variational approach to the $n$-body problem have led to significant progress in celestial mechanics, uncovering new types of possible orbits. In this paper, critical points of the Lagrangian action associated with the $n$-body problem are analysed using evolutionary algorithms to identify periodic and symmetrical solutions of the discretised system. A key objective is to locate minimum points of the action functional, as these correspond to feasible periodic solutions that satisfy the system's differential equations. By employing both stochastic and deterministic algorithms, we explore the solution space and obtain numerical representations of these orbits.  

Next, we examine the stability of these orbits by treating them as critical points. One approach is to compute their discrete Morse index to distinguish between minimum points and saddle points. Another is to classify them based on their action levels. Finally, analysing the boundaries of their attraction basins allows us to identify non-minimal critical points via the Ambrosetti-Rabinowitz Mountain Pass Theorem. This leads to an updated version of the algorithm that provides a constructive proof of the theorem, yielding new orbits in specific cases.  

This paper builds upon and extends the results presented in \cite{nostro}, providing a more detailed theoretical framework and deeper insights into the formulation. Additionally, we present new numerical results and an extended analysis of the critical points found, further enhancing the findings of the previous study.  
}

\vspace{2mm}

\tableofcontents

\section{Introduction}
\label{sec:intro}

The $n$-body problem examines the motion of $n$ massive bodies in the Euclidean space, where their trajectories are governed by mutual gravitational attractions. Let $\R^d$ represent the $d$-dimensional Euclidean space, and let $m_1, \ldots, m_n > 0$ denote the positive masses of these bodies, with with $n \geq 2$. The configuration space $\rchi$ of these heavy $n$ bodies with the centre of mass of the system fixed at the origin is defined as
\begin{equation*}
    \rchi \coloneq \left \{ x = (x_1, \ldots, x_n) \in (\R^d)^n : \sum_{i=1}^{n} m_ix_i = 0 \right \}.
\end{equation*}
If the $i$-th and $j$-th bodies collide, the $n$-body problem encounters a singularity. To address this, the singular sets are defined as
\begin{equation*}
    \Delta_{ij} \coloneq \left \{ x \in \rchi : x_i = x_j \right \} \quad \text{and} \quad \Delta \coloneq \bigcup_{i,j} \Delta_{ij},
\end{equation*}

\noindent allowing us to define the collision-free configuration set as $\hat{\rchi} \coloneq \rchi \setminus \Delta$. The potential function associated with the classical $n$-body problem is
\begin{equation*}
    U(x) \coloneq \sum_{i<j} \dfrac{m_im_j}{\|x_i -x_j\|}, \quad \text{for} \,\, x \in \rchi
\end{equation*}
and, given $T_x\rchi$ the tangent space to $\rchi$, the kinetic energy function is
\begin{equation*}
    K(\dot{x}) \coloneq \dfrac{1}{2} \sum_{i=1}^n m_i \|\dot{x}\|^2,  \quad \text{for} \,\, \dot{x} \in T_x\rchi,
\end{equation*}
while the Lagrangian function is
\begin{equation*}
    L(x, \dot{x}) \coloneq K(\dot{x}) + U(x),  \quad \text{for} \,\, x \in \rchi, \, \dot{x} \in T_x\rchi.
\end{equation*}

\noindent In this framework, the equations of motion for the bodies $x_1(t), \ldots, x_n(t)$ in configuration space $\rchi$ are given by Newton’s laws:

\begin{equation}
\label{bc 1}
    m_i \Ddot{x}_i = - \dfrac{\partial U}{\partial x_i}, \quad i = 1, \ldots, n.
\end{equation}
\\
\noindent There is considerable interest in identifying periodic solutions of Equation \eqref{bc 1}. In our framework, a periodic solution of \eqref{bc 1} implies that each trajectory \( x_i(t) \) of the \( i \)-th body traces a closed curve in \( \mathbb{R}^d \), often referred to as a loop within the configuration space.

\noindent Since the 1990s, variational methods have proven effective in generating collision-free solutions to \eqref{bc 1}. We refer the reader to the following articles and references therein: \cite{Ambrosetti2, Arioli, Bahri, Bessi, Chenciner2, ChencinerMontgomery, FerrarioTerracini05, Chen, Ramos, ChencinerVenturelli, Barutello_2004, gronchi, MONTALDI_STECKLES_2013, Sim}. Following the variational approach outlined in \cite{FerrarioTerracini05}, for \( T > 0 \), let us define the time circle \( \mathbb{T} \simeq \mathbb{R} / T \mathbb{Z} \), allowing us to set up the space of \( H^1 \)-periodic loops in \( \rchi \), as well as those free from collisions, as follows:
\begin{align*}
    \Lambda &\coloneq H^1(\mathbb{T}; \rchi) = \{ x \in H^1([0,T] ; \rchi) : x(0) = x(T) \}, \\
    \hat{\Lambda} &\coloneq H^1(\mathbb{T}; \hat{\rchi}) = \{ x \in H^1([0,T] ; \hat{\rchi}) : x(0) = x(T) \}.
\end{align*}

\noindent We define the Lagrange action functional \( \mathcal{A}: \Lambda \longrightarrow \mathbb{R} \cup \{ + \infty \} \) by
\begin{equation}
\label{action functional}
    \mathcal{A}(x) \coloneq \int_0^T L(x(t), \dot{x}(t)) \, \text{d}t, \quad \text{for} \,\, x \in \Lambda.
\end{equation}

\noindent As \( \mathcal{A} \in \mathcal{C}^1(\hat{\Lambda}) \), it follows that if \( x \in \hat{\Lambda} \) is a critical point of \( \mathcal{A} \), then \( x \) is a \( T \)-periodic solution of \eqref{bc 1}. However, in the context of the $n$-body problem, the first challenge is the lack of compactness in the open set of collision-free paths, which also affects the sub-levels of $\mathcal{A}$. Nonetheless, certain modifications can reduce the problem to more tractable models that restore coercivity, thereby enabling minimisation. We refer the reader to the following articles and references therein: \cite{gronchi, ChencinerVenturelli, FerrarioTerracini05, BarutelloFerrarioTerracini, XIA2022302}.
\\
\noindent Following \cite{FerrarioTerracini05}, in this work, we consider one such approach, in which candidate loops must meet specific symmetries in time, space, and indices. Specifically:
\begin{itemize}
    \item the \textit{space} \( \mathbb{R}^d \) where each component \( x_i(t) \) of \( x(t) \) resides;
    \item the \textit{time circle} \( \mathbb{T} \subset \mathbb{R}^2 \), representing the trajectory's period;
    \item the \textit{index set} \( \{1, \ldots, n\} \), which labels the \( n \)-bodies. \\
\end{itemize}

\noindent For a finite group \( G \), the action of \( G \) on the space of \( H^1 \)-periodic loops can be represented via three different group actions on \( \mathbb{R}^d \), \( \mathbb{R}^2 \), and \( \{1, \ldots, n\} \). Specifically, \( G \) acts as a subgroup of \( O(d) \), \( O(2) \), and the symmetric group \( \Sigma_n \) through the following homomorphisms:
\begin{equation*}
   \rho : G \rightarrow O(d), \quad \tau : G \rightarrow O(2), \quad \sigma : G \rightarrow \Sigma_n.
\end{equation*}

\noindent Define \( \bar{G} \) as \( \bar{G} \coloneq G / \mathrm{ker}\,\tau \). If \( \bar{G} \) is non-cyclic, let \( \I =[t_0, t_1] \) represent the fundamental domain (see \cite{FerrarioTerracini05}) for the action of \( G \), and denote \( H_0 \) and \( H_1 \) as the two maximal \( \mathbb{T} \)-isotropy subgroups corresponding to the instants \( t_0 \) and \( t_1 \), respectively. We define the space of fixed-end paths as
\begin{equation*}
    Y \coloneq \{x \in H^1(\mathbb{T}; \rchi^{\mathrm{ker}\,\tau}) : x(t_0) \in \rchi^{H_0}, \, x(t_1) \in \rchi^{H_1}\}.
\end{equation*}
If \( \bar{G} \) is cyclic, then we define \( Y \) as
\begin{equation*}
    Y \coloneq \{x \in H^1([t_0, t_1]; \rchi^{\mathrm{ker}\,\tau}) : g\,x(t_0) = x(t_1)\}.
\end{equation*}

\noindent The space of fixed-end paths $Y$ pointwise coincides with $\Lambda$  (or $\hat{\Lambda}$) and hence we can study the restriction of the action functional to the fundamental domain \( \I \) on \( Y \):
\begin{equation}
\begin{aligned}
\label{action functional on fundamental domain}
    \mathcal{A}_{\I} : Y &\longrightarrow \mathbb{R} \cup \{+ \infty \} \\
    y &\longmapsto \mathcal{A}_{\I}(y) \coloneq \int_{\I} L(y(t), \dot{y}(t)) \, \text{d} t.
\end{aligned}
\end{equation}

\noindent  Consequently, this reformulates the problem as a \textit{Bolza}-type problem, where the initial conditions reduce to conditions on the boundaries of $\I$, which are thus \textit{Bolza} conditions. Without loss of generality, we can assume \( \I = [0, \pi] \) and $T = l \pi$, with $l \in \N$. By leveraging the symmetry group \( G \), a critical point of \eqref{action functional on fundamental domain} can then be extended symmetrically, producing the entire orbit. 

\noindent This approach is essential for the minimisation problem, as it enables us to transition from analysing periodic loops to minimising among paths with fixed-end conditions, where initial and final points are constrained to the endpoints of such subdomains. For further details, refer to \cite{FerrarioTerracini05, symorb}. \\
This paper is structured as follows: Section \ref{sec:Finding the critical points} presents an algorithm for identifying critical points, which is then applied to the symmetrical $n$-body problem in Section \ref{sec: Solutions of the symmetrical $n$-body problem}; Section \ref{sec:Analysing the critical points} discusses two methods for analysing the identified critical points—using the discrete Morse index (Section \ref{sec: Discrete Morse index}) or intra and trans level distances (Section \ref{sec: Intra and trans level distances}). In Section \ref{sec:Mountain pass solutions}, an enhanced version of the algorithm that provides a constructive proof of the Ambrosetti-Rabinowitz Mountain Pass Theorem is introduced, which is subsequently applied to the symmetrical $n$-body problem in Section \ref{sec: Mountain pass solutions for the symmetrical $n$-body problem} to identify non-minimal critical points in specific cases. Moreover, in Section \ref{section Examples}, we present two examples that put into practice the theory developed in the previous sections.

\section{How to find critical points}
\label{sec:Finding the critical points}
Following \cite{CI24}, there are two possible approaches for discretising the action functional \( \mathcal{A} \) defined in \eqref{action functional}.

\paragraph{Discretisation through points}
\label{Discretisation through points}
The first approach explained in this paragraph involves discretising the integral operator using an appropriate quadrature technique and then approximating the derivatives using finite differences. Specifically, we consider a uniform grid \( 0 = t_0 < t_1 < \ldots < t_{M+1} = \pi \), where \( t_k \) represents the \( k \)-th time instant within the interval \( \I \). We define the function \( G(t) \) as follows:
\begin{equation*}
    G(t) = G(x_1(t), x_2(t), \ldots, x_n(t)) = \sum_{i=1}^n \frac{m_i}{2} \|\dot{x}_i\|^2 + \sum_{i<j} \frac{m_i m_j}{\|x_i - x_j\|}
\end{equation*}
Next, we apply the composite trapezoidal rule to approximate the integral over the \( M+1 \) intervals, yielding:
\[
\mathcal{A}_{\I} = \int_0^\pi G(t) \, dt \sim \frac{h}{2} \left[ G(t_0) + G(t_{M+1}) + 2 \sum_{k=0}^{M} G(t_k) \right]
\]
where \( h = \frac{\pi}{M+1} \), and \( y_i^k \sim x_i(t_k) \in \mathbb{R}^d \) for \( i = 1, \ldots, n \) and $ k = 0, \ldots, M+1 $. Using a forward difference formula for the approximation of the derivative with respect to time, we obtain the following approximation for the action functional:
\begin{equation}
\label{Approximated action functional}
\begin{aligned}
    \mathcal{A}_{\I} & \sim \sum_{k=1}^{M} \left[ \sum_{i=1}^n \frac{m_i}{2} \frac{\|y_i^{k+1} - y_i^k\|^2}{h} + h \sum_{i<j} \frac{m_i m_j}{\|y_i^k - y_j^k\|} \right] \\
    &+ \sum_{i=1}^n \frac{m_i}{4h} \left[ \|y_i^1 - y_i^0\|^2 + \|y_i^{M+2} - y_i^{M+1}\|^2 \right] \\
    &+ \frac{h}{2} \sum_{i<j} m_i m_j \frac{\|y_i^0 - y_j^0\| + \|y_i^{M+1} - y_j^{M+1}\|}{\|y_i^0 - y_j^0\| \|y_i^{M+1} - y_j^{M+1}\|} = \mathcal{A}^{(1)}_h
\end{aligned}
\end{equation}
\\
where we introduce appropriate ghost points \( y_i^{M+2} \). For further details, see \cite{turn0search0, turn0search3,turn0academia19, myrdal2013, CI24, symorb}.

\noindent As a result, we need to optimise the objective function:
\begin{equation} 
\label{eq:f1}
    f_1(y_i^k) = \mathcal{A}^{(1)}_h, \quad i = 1, \ldots, n, \; k = 0, \ldots, M+1
\end{equation}
The numerical solution is represented by a block real matrix \( (y_i^k)_{i,k} \in \mathbb{R}^d \), with \( i = 1, \ldots, n \) and \( k = 0, \ldots, M+1 \), having dimensions \( d \times n \times (M+2) \).

\paragraph{Discretisation through Fourier coefficients} 
\label{Discretisation through Fourier coefficients}
Another approach is using Fourier coefficients; for further details, see \cite{Sim, Sim2, Sim3, gronchi}. In our specific case following \cite{symorb}, one approximates the solution as the sum of a linear part, in particular the segment between the initial configuration $x_0$ and the final configuration $x_1$, and a truncated Fourier series, in details
\begin{equation}
\label{eq:F_tronc}
    x(t) = x_{0} + \frac{t}{\pi}(x_{1} - x_{0}) + \psi(t),
\end{equation}
where \( t \in [0,\pi] \) and \( \psi(0) = \psi(\pi) = 0 \).
\noindent We approximate the function \( \psi(t) \) using a truncated Fourier polynomial where $F \in \N$ stands for the length of Fourier polynomials  \( \psi(t) = \sum_{k=1}^F A_{k} \sin( k t) \in \mathbb{R}^{(d\times n)} \), for \( t \in [0,\pi] \). In this expression, we only use sine functions to ensure that the approximate solution remains periodic. Substituting the expression \eqref{eq:F_tronc} into the action functional \eqref{action functional} results in:
\begin{equation}
\label{eq:f2}
    f_2(x_0, x_1, A_{k}) = \mathcal{A}^{(2)}_h, \quad k = 1, \ldots, F.
\end{equation}
This is a function in \( (F+2) \times d  \times n \) variables which can be minimised using an appropriate optimisation technique. For further details, see \cite{symorb}.
\\

\noindent In both models, we introduce a scalar function \( f_i \), for \( i = 1, 2 \), given by either \eqref{eq:f1} or \eqref{eq:f2}, in the variables \( x \), represented by \( y_i^k \) or \( \{x_0, x_1, A_{k}\} \), which correspond to the unknowns in our problem. The functions \( f = f_i \), for \( i = 1, 2 \), to be maximised or minimised, are typically referred to as the objective function.

\subsection{Numerical optimisation routine}
\label{Numerical optimisation routine}
Following \cite{symorb}, we outline the following steps for the numerical optimisation routine:\\

\begin{enumerate}[label=\textbf{Step \Roman*}, leftmargin=*]
    \item Consider a random sequence of vectors
    \begin{equation*}
        A_0, A_1, \ldots, A_{F+1} \in \mathbb{R}^{(d \times n)},
    \end{equation*}
    where \( A_0 \) and \( A_{F+1} \) represent the initial and final points of the ansatz. Since this is a \textit{Bolza}-conditioned problem, project these vectors onto \( \rchi \) and then onto the boundary manifolds. For further details, see \cite{symorb}. \\
    
    \item Discretise the action functional defined in \eqref{action functional} using Fourier coefficients as outlined in Equation \eqref{eq:F_tronc}. From this, construct an initial guess \( y(t) \coloneq x(t) + s(t) \), where \( x(t) \) denotes the straight line segment between \( x_0 \) and \( x_1 \):
    \begin{equation*}
        x(t) = x_0 + \frac{t}{\pi}(x_1 - x_0), \quad \text{for} \,\, t \in [0,\pi],
    \end{equation*}
    and \( s(t) \) is the truncated Fourier series:
    \begin{equation*}
        s(t) \coloneq \sum_{k=1}^F A_k \sin(k t), \quad \text{for} \,\, t \in [0,\pi].
    \end{equation*}
    
    \item Discretise the interval \( [0,\pi] \) with a grid \( (t_h) \), and project any configuration \( y(t_h) \) onto the symmetric configuration space in order to satisfy the Bolza conditions of the problem.
    \\
    \item \label{Step 4: Choice of the algorithm} Express the action functional on the interval \( [0,\pi] \), evaluated at \( y \) as in expression \eqref{eq:f2}. This is now a function of \( (F+2) \times d \times n \) variables, which are the parameters to optimise. The choice of optimisation algorithm is discussed in the following paragraph, but one could, for instance, use a \textit{steepest descent method} or a \textit{Newton method} to move from \( y \) to \( y_{\text{new}} \). Once this is done, project \( y_{\text{new}} \) back into the symmetric configuration space.
    \\
    \item The optimised path, now defined on the fundamental domain \( \I \) through Fourier coefficients, can be extended to the entire period using the symmetries of the group and an inverse Fourier transform.
    
\end{enumerate}

\paragraph{Choice of the optimising algorithm}

In this paragraph, we discuss the selection of the algorithm outlined in \ref{Step 4: Choice of the algorithm}. It is important to note that there is no single ``universal" algorithm for solving a particular problem. Instead, different algorithms are tailored to specific optimisation problems based on their characteristics. We refer the reader to the following article and references therein: \cite{CI24, Nocedal, Kapela2007, hillier2005, Akan2023, Gupta2021, Liu2014, Meng2005, OGUNTOLA2021109165, Oztas2023, Kapela2003}.

\noindent Optimisation algorithms are iterative methods, and they can be distinguished by the technique they use to progress from one iteration to the next. Broadly speaking, we can categorise methods into those that compute the next iteration using the values of the objective function, the constraint functions, and potentially their first or second derivatives. Other methods, however, rely on the information accumulated from previous iterations or the local information derived from the current point.

\noindent Choosing the right algorithm is a critical decision, as it impacts both the speed at which a problem is solved and the quality of the solution. A good algorithm should be efficient in terms of memory usage and computational time, but it must also be accurate. That is, it must deliver precise solutions without being overly sensitive to data errors or rounding issues that arise when the algorithm is implemented on a computer. However, it is important to remember that these goals can often compete with one another.

\noindent Furthermore, the regularity of the functions involved is key. It helps exploit local information about constraints and objective functions, making it easier to assess the nature of the solutions, whether local or global. Many optimisation algorithms for nonlinear problems lead to local minimising solutions, where the objective function is smaller than at nearby points. Finding a global solution, where the objective function attains its smallest value across the entire domain, is often challenging. In cases where the objective function is convex, global solutions are also local. However, for most nonlinear problems, whether constrained or unconstrained, solutions tend to be local and it is often necessary to search multiple regions of the solution space before applying global optimisation methods.

\noindent In \cite{CI24}, several novel approaches are introduced to tackle the challenge of identifying as many critical points as possible of the action functional defined in \eqref{Approximated action functional}. Given the complexity of the problem, this paper adopts a hybrid strategy combining two complementary algorithms: a multi-population adaptive version of inflationary differential evolution (MP-AIDEA, see \cite{DVM20}) and a domain decomposition method referred to as Lattice. MP-AIDEA integrates differential evolution with restart and local search mechanisms from monotonic basin hopping, enabling an adaptive and efficient exploration of the solution space. Meanwhile, the Lattice algorithm ensures thorough domain decomposition, systematically covering the entire space. Together, these methods enhance the ability to uncover periodic orbits that satisfy the system's differential equations.  

\noindent A key advantage of MP-AIDEA over traditional techniques like gradient descent is its ability to navigate highly non-convex optimisation landscapes. Due to the intrinsic nature of the problem, one expects an exceptionally high number of local minimum points and saddle points, resulting in a complex and structured solution space. Gradient descent, as a single-point method, tends to become trapped in individual local minimisers, limiting its ability to explore the broader landscape.  

\noindent In contrast, MP-AIDEA employs a multi-population framework with adaptive mechanisms that enable a more global and systematic search. By integrating differential evolution with local search strategies and restart mechanisms, it efficiently escapes shallow local minimisers and captures a diverse set of solutions. This adaptability is particularly valuable in our setting, where the abundance of minimum points suggests a rich network of periodic orbits.  

\noindent Additionally, MP-AIDEA’s parallelised structure and integration with the Lattice method enhance its robustness and efficiency in high-dimensional spaces. This approach ensures better scalability and a greater diversity of solutions compared to conventional optimisation techniques, making it a powerful tool for systematically identifying multiple periodic orbits and uncovering the deeper structure of the problem.

\subsection{MP-AIDEA}
\label{MP-AIDEA}
Some optimisation problems are so complex that finding an exact optimal solution is not feasible. To address this, metaheuristic methods have been developed to provide strategic guidelines and flexible frameworks for exploring solution spaces. These methods are particularly useful for problems where traditional optimisation techniques struggle, as they balance exploration (searching broadly across the solution space) and exploitation (refining promising solutions). For more details on metaheuristic methods, see \cite{CI24, Gupta2021, Parejo2012, Sorensen2012}. One example of such an approach is the Multi Population Adaptive Inflationary Differential Evolution Algorithm (MP-AIDEA, \cite{DVM20}).

\noindent MP-AIDEA is a population-based evolutionary algorithm designed to find approximate solutions for single-objective global optimisation problems in continuous spaces. In this context, a population refers to a group of candidate solutions, called agents, that evolve over iterations to find better solutions. Let $N_\text{P}$ be the number of populations and $N_\text{A}$ be the number of agents in each population. MP-AIDEA operates through the following main steps: \\
\begin{enumerate}
    \item \textbf{Differential Evolution}: MP-AIDEA starts by running $N_\text{P}$ instances of Differential Evolution (DE \cite{Price2005}) in parallel, one for each population. Differential Evolution is an optimisation technique that generates new candidate solutions by combining existing ones. Initially, agents are randomly distributed across the parameter space to ensure broad exploration. Each DE process consists of $N_\text{S}$ stages, where new parameter vectors are generated through three main operations:
    \begin{itemize}
        \item \textit{Mutation}: randomly combining parameter vectors to create a new candidate solution.

    \item \textit{Crossover}: mixing components of the new candidate with the current agent to increase diversity.
    
    \item \textit{Selection}: comparing the new candidate with the current agent and keeping the better one.
    \end{itemize}
    
\noindent This cycle repeats until a stopping criterion is reached. For more details, see \cite{Price2005, DVM20, CI24}. \\

    \item \textbf{Local Search}: after the Differential Evolution phase, a local search is performed to refine the solutions. This search starts from the best agent in each population but only if that agent is not in the basin of attraction of a previously found local minimum points. Here, the basin of attraction is the region of the solution space that leads to a particular local minimum point.\\

    \item \textbf{Basin Hopping}: to explore different regions of the solution space, each population is restarted near the local minimum point found in the previous step. This technique, known as Basin Hopping, allows the algorithm to "jump" from one basin of attraction to another, potentially finding better local minimum points. For more details, see \cite{Leary2000, Wales1997}. \\
\end{enumerate}

\noindent By integrating Differential Evolution with local search and Basin Hopping, MP-AIDEA efficiently explores the solution space. This combination allows populations to move within a funnel structure (see \cite{funnel_structure}), transitioning from one local minimum point to another until a global minimum point is found.

\noindent MP-AIDEA introduces a unique mechanism to avoid repeatedly detecting the same local minimum point. When a population revisits the basin of attraction of an already found minimum point, it is restarted in a different region of the solution space. To track discovered minimum points, an archive $\textgoth{A}$ is maintained. This archive stores all detected local minimum points along with the population states at each restart, ensuring a diverse exploration of the solution space without redundancy. For more details, see \cite{DVM20}.

\subsection{Lattice}
\label{Lattice}

\noindent The main goal is to explore the domain in a systematic way to find as many local minimum points as possible. To do this, we use the Lattice algorithm, which decomposes the domain into smaller regions based on the distribution of local minimum points identified so far. This adaptive partitioning ensures that the search becomes more focused and efficient as we progress. 

\noindent  Initially, the Lattice algorithm requires a set of known local minimum points, which are obtained by running an optimisation algorithm. Once these minimisers are identified, they guide the domain decomposition process. At each decomposition step, referred to as a \textit{layer} \( j \), where \( j = 1, \ldots, \text{NoL} \in \mathbb{N} \), the domain is divided into two subregions. At this point, the optimisation algorithm is rerun within each subregion to find a new set of minimisers, which are then used to guide the next step of the decomposition. This iterative partitioning allows for a progressively finer search, with each subregion being explored independently. We choose MP-AIDEA as the optimisation algorithm for this task, but other optimisation algorithms could also be used for this purpose.

\noindent More specifically, let \( \Xi \) represent the domain of our optimisation parameters. The challenging aspect lies in determining how to divide \( \Xi \) into two subregions, \( \Xi_1 \) and \( \Xi_2 \). A simple equal division is not optimal, as we aim to focus on areas where fewer minimum points have been found. Thus, we prefer an unequal division, with one subregion containing a higher concentration of minimum points, while the other remains less populated.

\noindent This approach can be demonstrated with a simple example.
\vspace{0.2cm}

\begin{ex}
\label{Example lattice}
The process is repeated for each layer, so we consider \( j=1 \) for this example.\\
Let \( a, b, c, d \in \mathbb{R} \), where \( a < b \) and \( c < d \). Define the optimisation parameters \( x \in I = [a, b] \subset \mathbb{R} \) and \( y \in J = [c, d] \subset \mathbb{R} \), so that our domain is \( \Xi = I \times J \). This can be easily generalized to \( n \in \mathbb{N}, n > 2 \) optimisation parameters.

\noindent Let \( \mathcal{M} = \{m_i\}_{i=1}^M \), with \( M < +\infty \), be the set of local minimum points found during the optimisation process. Each local minimum point \( m_i \) corresponds to a point in the optimisation domain where the function reaches a local minimum value. Specifically, for each \( m_i \), we have \( m_i = (x_i, y_i) \in I \times J \) for \( i = 1, \dots, M \), where \( m_i \) is the local minimum point in the domain. Our goal is to examine how the set \( \mathcal{M} \), and the corresponding local minimum points \( m_i \), are distributed across the domain \( \Xi \), which in this example is \( I \times J \).

\noindent From this setup, we proceed as follows:

\begin{enumerate}
    \item we analyse the set \( \mathcal{M} \) to determine how well the points \( m_i \) follow a uniform distribution. This is done using a chi-square test with respect to the optimisation parameters \( x, y \);
    \item we identify the optimisation parameter with the smallest p-value, indicating the one that deviates most from a uniform distribution. If there are ties, we choose the first one. For this example, let this parameter be \( y \). Therefore, the points \( m_i \) exhibit a non-uniform distribution with respect to the variable \( y \);
    \item to create an uneven division, we calculate the 70th percentile of the optimisation parameter found in step 2, which in this example would give a specific value \( \bar{y}_m \in J \);
    \item we divide the domain of the optimisation parameter \( y \) based on the value \( \bar{y}_m \) calculated in step 2. In this case:
    \[
    J = J_1 \cup J_2 \quad \text{with} \, \, J_1 = [c, \bar{y}_m] \quad \text{and} \, \, J_2 = [\bar{y}_m, d];
    \]
    \item finally, we divide the entire domain \( \Xi \) into two subregions using the division from step 4. For our example \( \Xi = I \times J = [a, b] \times [c, d] \), we define the new subregions as:
    \[
    \Xi_1 = I \times J_1 \quad \text{and} \quad \Xi_2 = I \times J_2.
    \]
\noindent   Note that \( \Xi_1 \cup \Xi_2 = \Xi \), ensuring that no part of the domain is excluded during the Lattice algorithm. \\
\end{enumerate}
Figure \ref{fig: Example Lattice} shows a possible outcome of this example where \( \text{NoL} = 3 \).

\begin{figure}[H]
    \centering
    \begin{subfigure}[b]{0.61\textwidth}
        \centering
        \includegraphics[width=\textwidth]{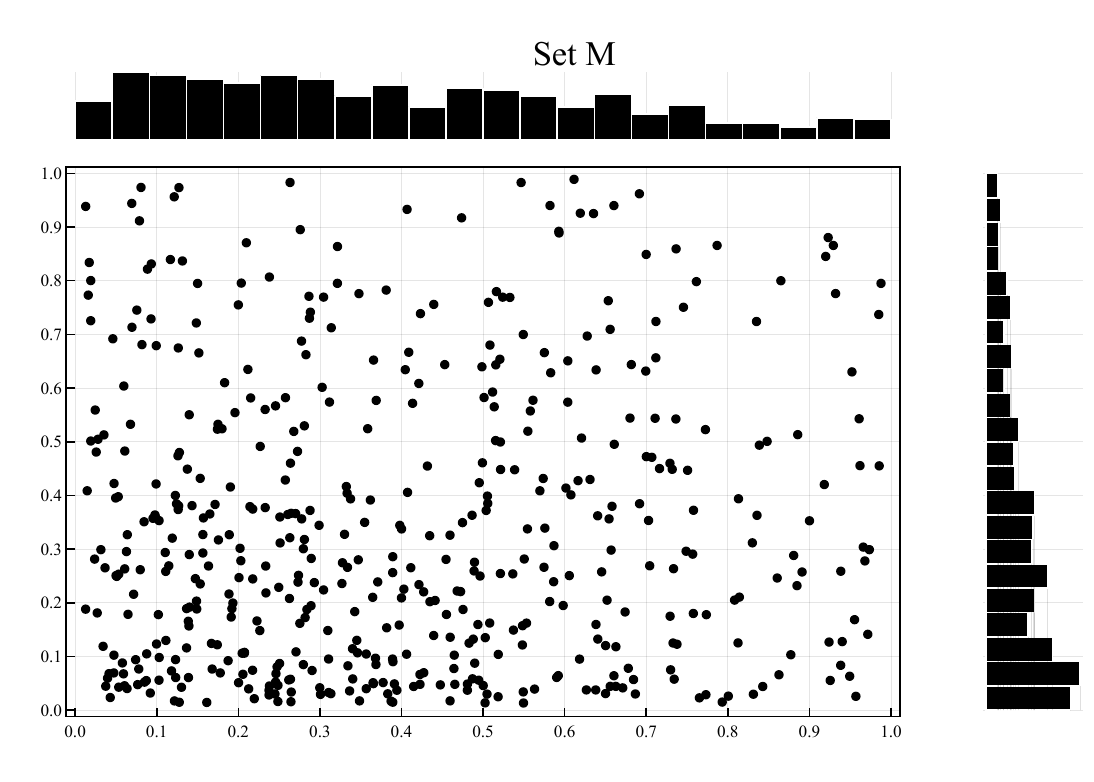}
    \end{subfigure}

    \vfill
    \begin{subfigure}[b]{0.47\textwidth}
        \centering
        \includegraphics[width=\textwidth]{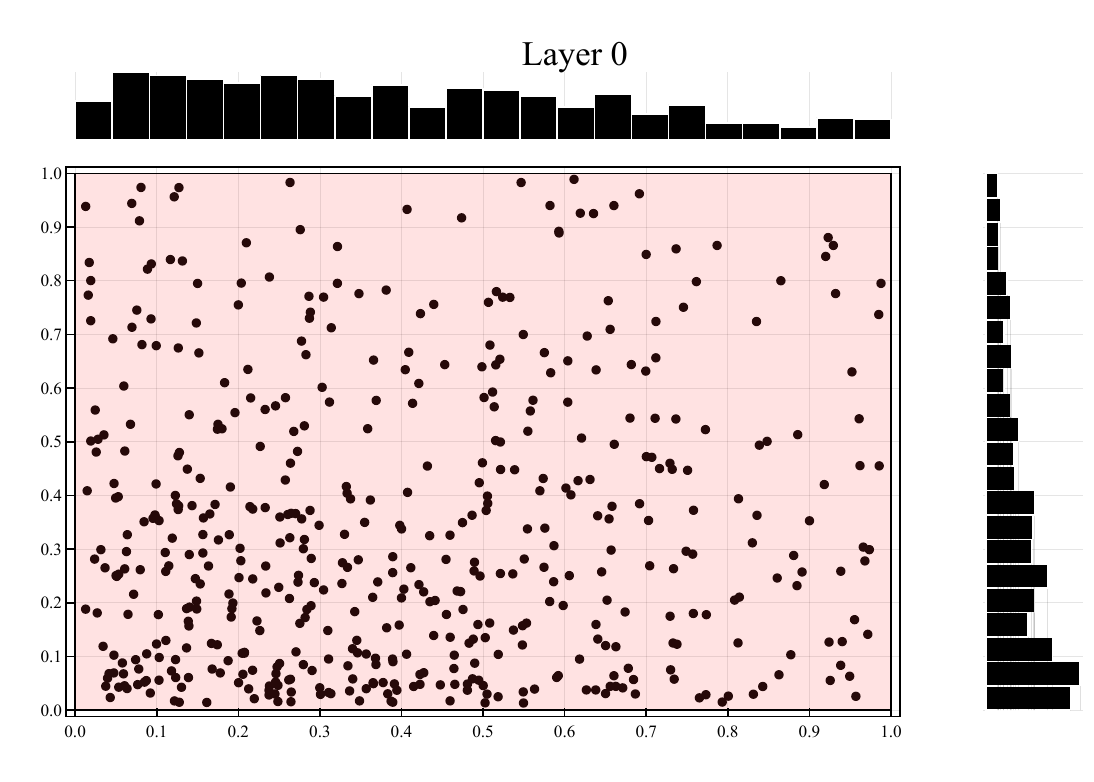}
    \end{subfigure}
    \hfill
    \begin{subfigure}[b]{0.47\textwidth}
        \centering
        \includegraphics[width=\textwidth]{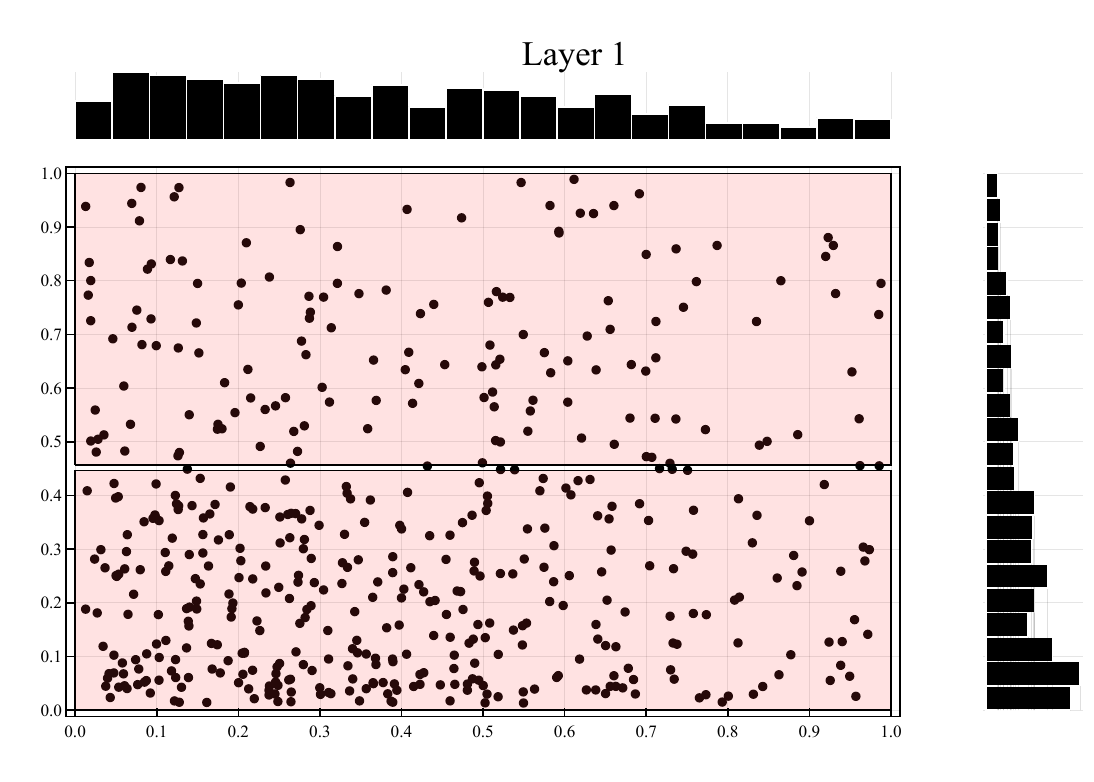}
    \end{subfigure}

    \vfill
    \begin{subfigure}[b]{0.47\textwidth}
        \centering
        \includegraphics[width=\textwidth]{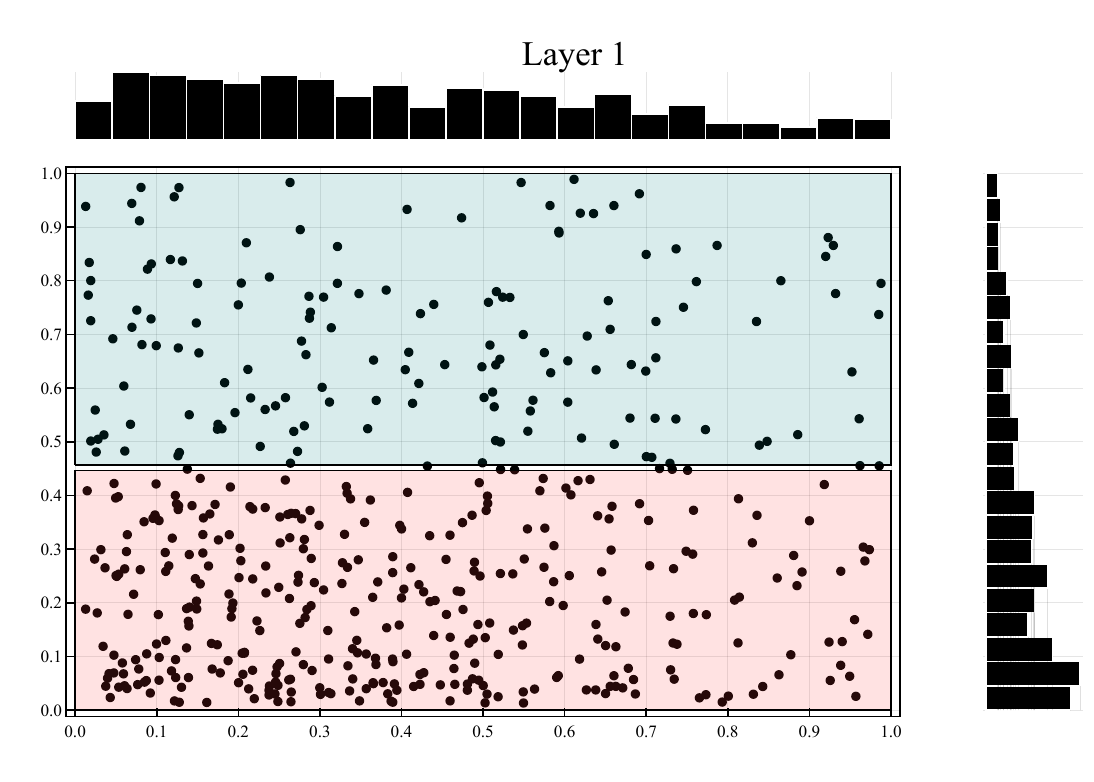}
    \end{subfigure}
    \hfill
    \begin{subfigure}[b]{0.47\textwidth}
        \centering
        \includegraphics[width=\textwidth]{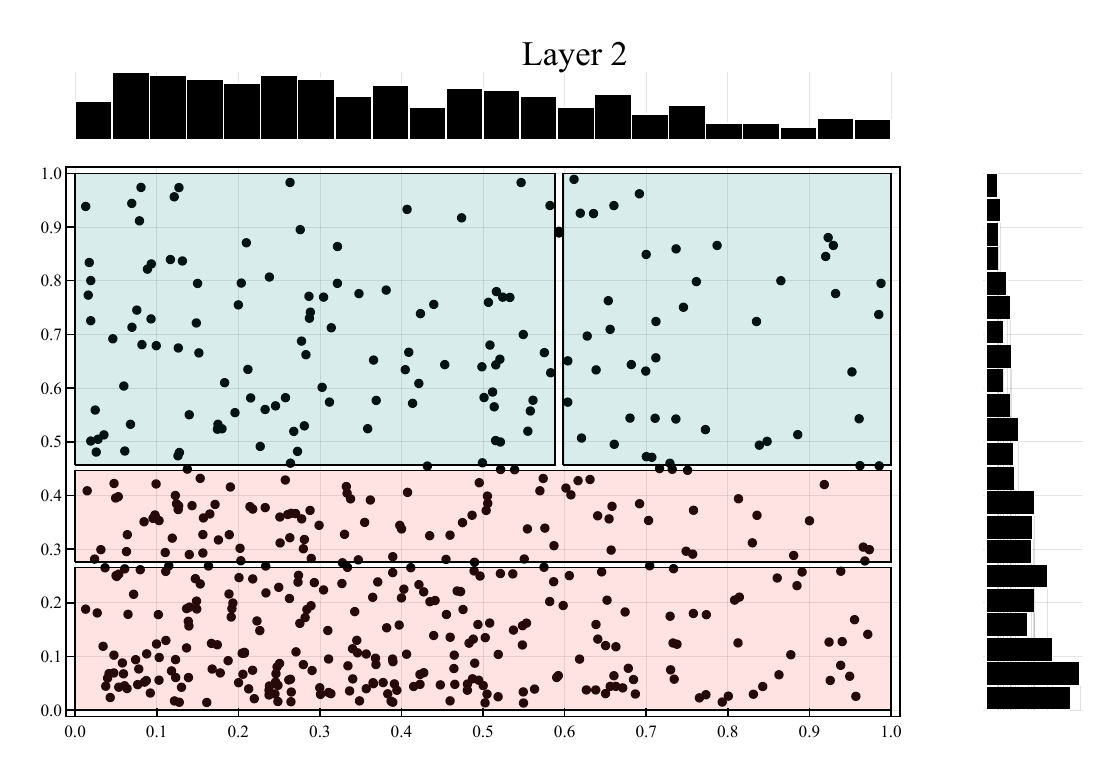}
    \end{subfigure}

    \vfill
    \begin{subfigure}[b]{0.48\textwidth}
        \centering
        \includegraphics[width=\textwidth]{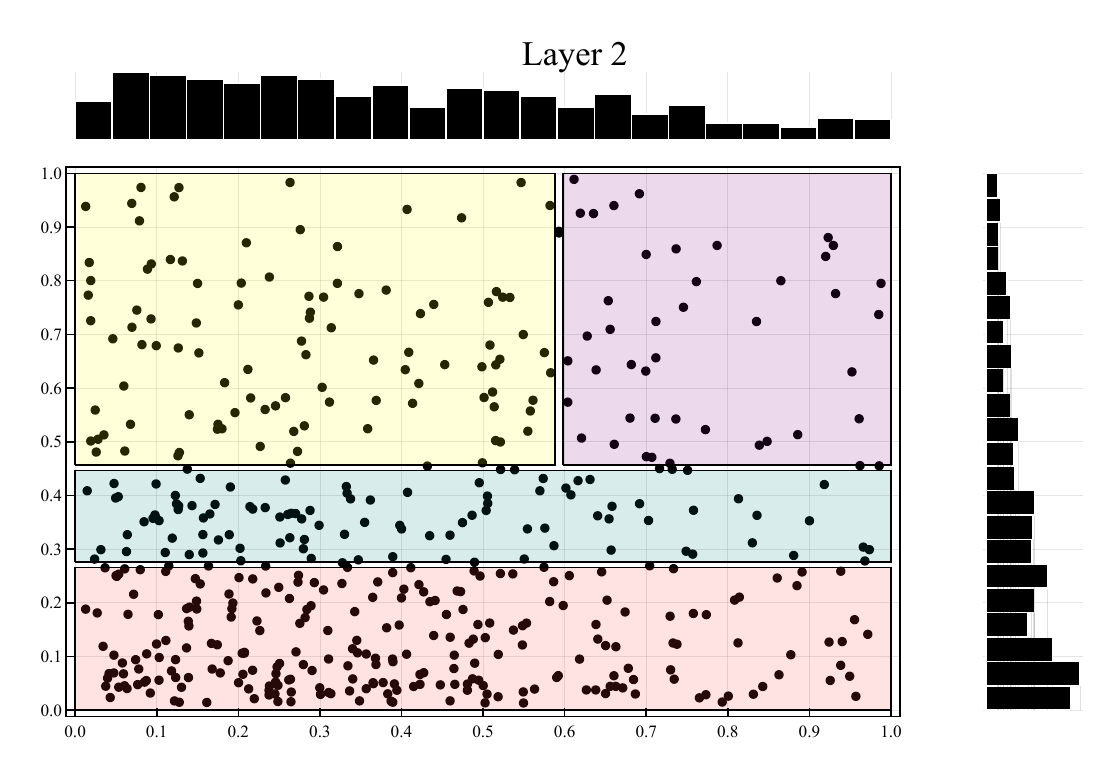}
    \end{subfigure}
    \hfill
    \begin{subfigure}[b]{0.48\textwidth}
        \centering
        \includegraphics[width=\textwidth]{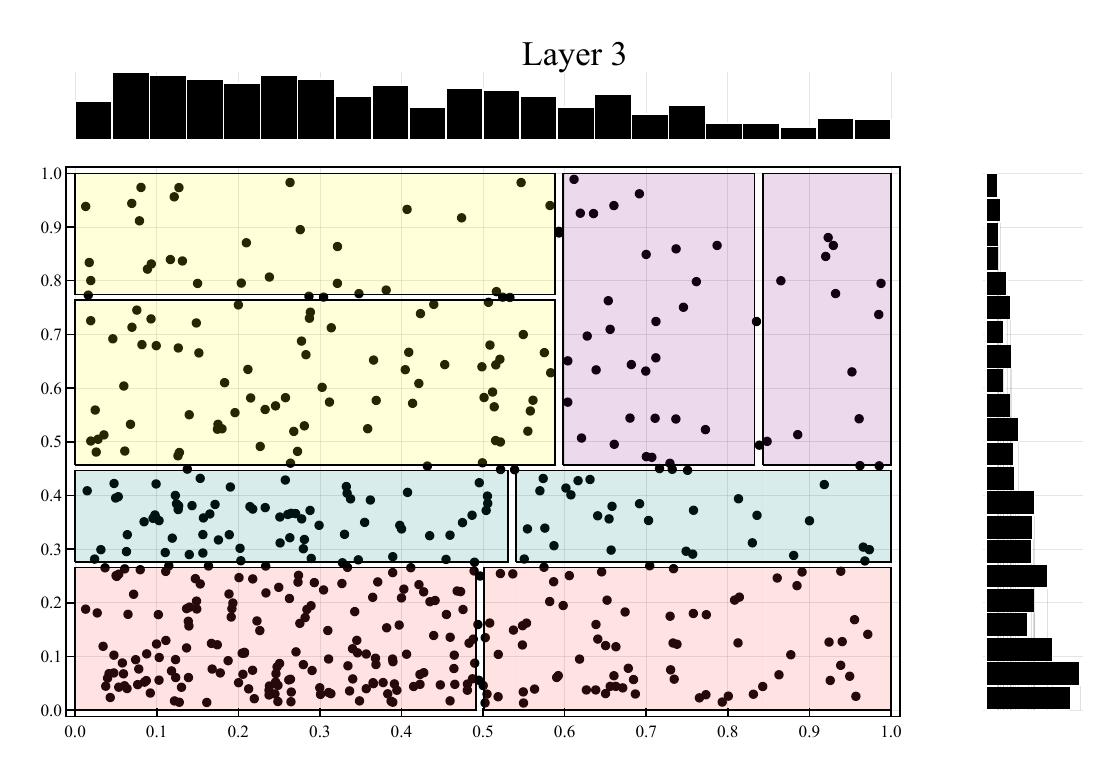}
    \end{subfigure}

    \caption{Visualisation of a two-dimensional example of the Lattice algorithm. The first figure displays the set $M$ along with the distributions of the variables. The subsequent images illustrate the progressive division of the domain at each layer, following the process described in Example \ref{Example lattice}.}
    \label{fig: Example Lattice}
\end{figure}

\end{ex}

\noindent The main concept behind the Lattice algorithm is summarized in Figure~\ref{fig: Lattice scheme}, which illustrates the first two layers for simplicity:

\begin{figure}[H]
\centering
\includegraphics[scale=0.6]{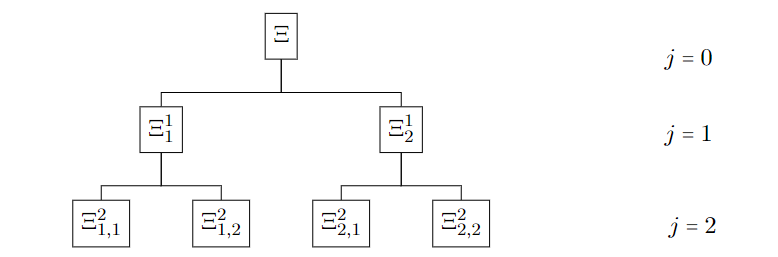}
\vspace{0.2 cm}
\caption{\scriptsize{Lattice Scheme}}
\label{fig: Lattice scheme}
\end{figure}

\vspace{-0.25 cm}

\noindent The Lattice algorithm is shown in Algorithm~\ref{alg: lattice}, where it is implemented as a recursive function. Algorithm~\ref{alg: One layer} provides the steps to perform the division for just one layer.

\begin{algorithm}[H]
    \caption{One layer Algorithm}
    \label{alg: One layer}
    \begin{algorithmic}[1]
    \Function{one\_layer\_function}{$\Xi$}
        \State \textbf{Input:} A domain $\Xi$.
        \State \textbf{Output:} Two subregions $\Xi_1$ and $\Xi_2$, with $\Xi_{1} \bigcup \Xi_{2} = \Xi$.
        \State\textbf{Step 1.} Find the p-value ($p$) and column index (min\_p\_index) for the column with the minimum p-value in domain $\Xi$. If multiple minimum p-values exist, pick first.
        \State\textbf{Step 2.} Calculate the 70th percentiles for the column with the minimum p-value
        \State\textbf{Step 3.} Divide min\_p\_index-column into two at the 70th percentiles, calling the new two columns $\Xi_1$ and $\Xi_2$.
        \State\textbf{Step 4.} \Return{$\Xi_1$ and $\Xi_2$}.
    \EndFunction
    \end{algorithmic}
\end{algorithm}

\vspace{-0.25 cm}

\begin{algorithm}[H]
\caption{Lattice Algorithm}
\label{alg: lattice}
\begin{algorithmic}[1]
\Procedure{lattice\_function}{$\Xi$, \text{NoL}, Output}
    \If{$j = \text{NoL}$}
        \State \Return{Output}
    \Else
        \State\textbf{Step 1.}  $\Xi_1$, $\Xi_2$ $\gets$ OUTPUT $\{$\Call{one\_layer\_function}{$\Xi$}$\}$
        \State\textbf{Step 2.}  Output[$j$] $\gets$ OUTPUT $\{$\Call{lattice\_function}{$\Xi_1$, $j + 1$, Output} $\cup$ \Call{lattice\_function}{$\Xi_2$, $j + 1$, Output}$\}$
        \State\textbf{Step 3.}  $j \gets j + 1$
        \State\textbf{Step 4.}  \Return{Output}
    \EndIf
\EndProcedure
\end{algorithmic}
\end{algorithm}

\section{Solutions of the symmetrical $n$-body problem}
\label{sec: Solutions of the symmetrical $n$-body problem}

The combination of the MP-AIDEA algorithm (described in Section \ref{MP-AIDEA}) and the Lattice algorithm (outlined in Section \ref{Lattice}) is highly effective. In particular, to execute the Lattice algorithm, one needs a set of possible local minimum points, denoted as \( \mathcal{M} = \{m_i\}_i^M \), with \( M < + \infty \). Rather than running an algorithm multiple times to construct this hypothetical set of minimum points, we take advantage of the features of MP-AIDEA. Specifically, we use the archive \( \textgoth{A} \), which stores all the local minimum points found, as well as the individuals in the population at each restart. Thus, we set \( \mathcal{M} = \textgoth{A} \).

\noindent In this case, the combined MP-AIDEA \& Lattice algorithm is implemented as a recursive function, as shown in Algorithm~\ref{alg:MP-AIDEA_Lattice}. Algorithm~\ref{alg: One generation Algorithm} describes the procedure for performing one generation of the algorithm.

\begin{algorithm}[H]
    \caption{One generation Algorithm}
    \label{alg: One generation Algorithm}
    \begin{algorithmic}[1]
    \State \textbf{Step 0.} Set the number of layers for each generation to $1$, $\text{NoL}= 1$. \\
    Define MP-AIDEA initialisation parameters $\text{MP-AIDEA\_param}$:
    \begin{itemize}
        \item cost function $f$;
        \item the optimisation parameters; set $\beta$ as the number of optimisation parameters;
        \item $\Xi \subset \mathbb{R}^{\beta}$, the domain of optimisation parameters;
        \item MP-AIDEA options.
    \end{itemize}
     \Function{one\_generation\_function}{$\text{MP-AIDEA\_param}, \Xi$}
    \State \textbf{Step 1.} Set initial values for MP-AIDEA's parameters in $\Xi$.
    \State \textbf{Step 2.} Run MP-AIDEA Algorithm and set $\mathcal{M} = \textgoth{A}$.
    \State \textbf{Step 3.} Analyse the set $\mathcal{M}$ to collect all critical points. Hence, given a tolerance $\text{tol} \ll 1$, 
    \begin{equation}
    \label{eq: for minimum}
    m_i \text{ is considered a critical point of } f \text{ if } \nabla f(m_i) \leq \text{tol}. 
    \end{equation}
    \State \textbf{Step 4.} Run the Lattice algorithm on the set $\mathcal{M}$, therefore set $\Xi = \mathcal{M}$ 
    \begin{equation*}
        \Xi_1, \Xi_2 \gets \text{OUTPUT} \{\Call{one\_layer\_function}{\Xi}\}\}
    \end{equation*} 
    and find $\Xi_1$ and $\Xi_2$.
    \State\textbf{Step 5.} \Return{MP-AIDEA\_Output, $\Xi_1$ and $\Xi_2$}.
    \EndFunction
    \end{algorithmic}
\end{algorithm}

\begin{algorithm}[H]
    \caption{MP-AIDEA \& Lattice Algorithm}
    \label{alg:MP-AIDEA_Lattice}
    \begin{algorithmic}[1]
    \State \textbf{Step 0.} Set $\text{NoL} \in \mathbb{N}$ as the number of generations and set the number of layers for each generation, $\text{NoL}= 1$. Define MP-AIDEA initialisation parameters $\text{MP-AIDEA\_param}$, therefore:
    \begin{itemize}
        \item[(i)] cost function $f$;
        \item[(ii)] the optimisation parameters; set $\beta$ as the number of optimisation parameters;
        \item[(iii)]\label{MP-AIDEA 3} $\Xi \subset \mathbb{R}^{\beta}$, the domain of optimisation parameters;
        \item[(iv)] MP-AIDEA options.
    \end{itemize}
\Procedure{MP-AIDEA\_lattice\_function}{$\Xi$, $\text{MP-AIDEA\_param}$, $\text{NoG},$ Output}
    \If{$j = \text{NoG}$}
        \State \Return{Output}
    \Else
        \State\textbf{Step 1.} Run \Call{one\_generation\_function}{MP-AIDEA\_param, $\Xi$} and find as output 
        \begin{itemize}
            \item MP-AIDEA\_Output;
            \item $\Xi_1$;
            \item $\Xi_2$.
        \end{itemize}
        Set $\tilde{\Xi} \coloneq \Xi_1 \bigcup \Xi_2$. \\
        In $\text{MP-AIDEA\_param}$ (i)-(iv), update the domain, MP-AIDEA's parameter number (iii), with $\tilde{\Xi}$.
        \State\textbf{Step 2.}  Output[$j$] $\gets$ OUTPUT $\{$\Call{MP-AIDEA\_lattice\_function}{$\tilde{\Xi}$, \text{MP-AIDEA\_updated\_param} $j$, Output}
        \State\textbf{Step 3.}  $j \gets j + 1$
        \State\textbf{Step 4.}  \Return{Output}
    \EndIf
\EndProcedure
    \end{algorithmic}
\end{algorithm}

\begin{obs*}
In this initial approach, we do not consider any exit flag for the MP-AIDEA algorithm. Note that \textbf{Step 3} of Algorithm \ref{alg: One generation Algorithm} is crucial, as it is possible that some subregions of the domain do not contain actual minimum points, even though the archive \( \textgoth{A} \) is not empty.
\end{obs*}

\begin{obs*}
Furthermore, for each layer processed by MP-AIDEA, the subregions are independent of one another. This allows the algorithm to run in parallel across layers, thus reducing the overall runtime.
\end{obs*}

\section{Further analysis of critical points}
\label{sec:Analysing the critical points}

Next, we focus on evaluating the functional stability of the problem. The orbit trajectories identified in the previous phase are treated as critical points, which allows us to analyse the stability or instability in their respective neighbourhoods. In Section \ref{sec: Discrete Morse index}, we explore a method to determine the nature of the critical points found in Section \ref{sec:Finding the critical points} by calculating their discrete Morse index, as proposed in the Section \textit{A study on the discrete Morse index and the Floquet index of the equivariant gravitational n-body problem} in \cite{IP24}. Additionally, in Section \ref{sec: Intra and trans level distances}, we present a graphical representation of the problem's structure using \textit{intra-level} distances, \( \rho_{\mathrm{IL}} \), and \textit{trans-level} distances, \( \rho_{\mathrm{TL}_k} \), between the local minimum points, as introduced in \cite{Vasile2}.

\noindent It is worth noting that another approach to analysing the local minimality of the solutions found in this variational setting is through conjugate point theory. This method allows one to determine whether the solutions are local minimisers of the action functional over the space of periodic loops. For instance, in \cite{Fenucci2022}, conjugate point theory was applied to study the local minimality properties of circular orbits in $1/r^\alpha$ potentials and the figure-eight solution of the three-body problem. While we do not employ this approach in our current analysis, the results in \cite{Fenucci2022} align with the ones presented in this paper, providing a broader perspective on the study of local minimisers in variational problems.

\subsection{Discrete Morse index}
\label{sec: Discrete Morse index}

Following \cite{IP24}, we define:

\begin{defn}[Computational discrete Morse index]
\label{def 2: Discrete Morse index}
    The \textit{index} of a non-degenerate critical point \( p \) of \( f: H \to \mathbb{R} \) of class \( \mathcal{C}^2 \) is the dimension of the largest subspace of the tangent space to \( M \) at \( p \) on which the Hessian is negative definite. Therefore, the \textit{discrete Morse index} \( \tilde{n}^{-} \) of a critical point \( p \) is defined as the number of negative eigenvalues of the Hessian matrix \( H_f(p) \).
\end{defn}

\noindent This definition can now be applied to the critical points identified in Section \ref{sec:Finding the critical points}. The discrete Morse index can be computed either on the fundamental domain \( \mathbb{I} \) or, after constructing the orbit using its symmetries, on the entire domain. Clearly, the index calculated on the fundamental domain will be less than or equal to the index calculated on the entire orbit.

\paragraph{\textit{On the fundamental domain}} Following \cite{symorb} and Section \ref{Discretisation through Fourier coefficients}, the Hessian matrix of the action functional can be computed on the fundamental domain \( \mathbb{I} \) using discretisation via Fourier coefficients, as described in \eqref{eq:f2}. To do so, we approximate the solution as in equation \eqref{eq:F_tronc} and substitute this into the action functional defined in \eqref{action functional}, leading to the expression in \eqref{eq:f2}. 

\paragraph{\textit{On the entire orbit}} Alternatively, to calculate the discrete Morse index on the entire orbit, the Hessian matrix of the action functional must be written differently. Following \cite{IP24} and Section \ref{Discretisation through points}, we compute the Hessian matrix using discretisation through points, as in \eqref{Approximated action functional}, with specific considerations regarding the initial and final points. Since the initial and final points of a periodic orbit coincide, only the initial point is considered. \\

\noindent In both cases, the Morse index of a critical point is the number of negative eigenvalues of the Hessian matrix calculated in that point. Therefore, each critical point (i.e., minimum points and saddle points) is assigned to its discrete Morse index. The indices for minimum points and saddle points are \( \tilde{n}^{-}(p) = 0 \) and \( \tilde{n}^{-}(p) > 0 \), respectively.

\subsection{Intra and trans level distances}
\label{sec: Intra and trans level distances}
\noindent In this section, we categorise the minimum points found in Section \ref{sec:Finding the critical points} based on the value of their objective function, which in our case is the action functional, following the approach in \cite{Vasile2}. We refer to these categories as \textit{levels}. Each level groups local minima with very similar action values, providing a structured way to analyse their distribution.  

\noindent For each level, we define the \textit{intra-level} distance \( \rho_{\mathrm{IL}} \) as the average of the relative distances between each local minimum and all other minima within the same level. Additionally, the \textit{trans-level} distance \( \rho_{\mathrm{TL}_k} \) quantifies the average relative distance between each local minimum and all other minima in the next lower level. For the lowest level, \( \rho_{\mathrm{TL}_k} \) is computed as the average distance from the best-known solution.  

\noindent We now provide a more precise definition of these concepts.

\begin{defn}[Intra and trans level distances]
\label{def: Intra and trans level distances}
Let \( \mathcal{M} = \{x_k\}_{k=1}^M \), with \( M < + \infty \), be the set of possible local minimum points found during the optimisation step. For each \( k \), we consider the value of the objective function in each minimum point \( f_k = f(x_k) \); then, let \( \varepsilon > 0 \) be a suitable small number. We build an interval \( I_k = [f_k - \varepsilon, f_k + \varepsilon] \). For each interval \( I_k \), we have the following set:
\begin{equation*}
    \mathfrak{m}_{I_k} \coloneq \{x_i \in \mathcal{M} \mid f(x_i) \in I_k \}.
\end{equation*}
We define the \textit{intra-level} distance \( \rho_{\mathrm{IL}} \) of a minimum point \( x_i \) as
\begin{equation}
    \label{eq: Intra level distance}
    \rho_{\mathrm{IL}}(x_i) \coloneq \dfrac{1}{|\mathfrak{m}_{I_i}|}\,\, \sum_{\substack{x_j \in \, \mathfrak{m}_{I_i} \\ x_j \neq x_i}} 
    \|x_i - x_j\| \,\, \text{with} \,\, x_i \in \mathfrak{m}_{I_i}.
\end{equation}
Then, we order the set \( \{f_k\}_k \) such that \( f_{k-1} < f_k < f_{k+1} \), and define the \textit{trans-level} distance \( \rho_{\mathrm{TL}_k} \) of a minimum point \( x_i \) as
\begin{equation}
\begin{small}
    \label{eq: Trans level distance}
        \rho_{\mathrm{TL}_k}(x_i) \coloneq \begin{dcases} \dfrac{1}{|\mathfrak{m}_{I_k}|}\, \sum_{x_j \in \, \mathfrak{m}_{I_k}} \hspace{-0.8em}
    \|x_i - x_j\|, \quad x_i \in \mathfrak{m}_{I_{k-1}} \,\, \text{for} \,\, k > 1 ,\\
    \dfrac{1}{|\mathfrak{m}_{I_1}|}\, \sum_{x_j \in \, \mathfrak{m}_{I_1}} \hspace{-0.7em}
    \|x_1 - x_j\|,\quad x_1 \,\, \text{is best-known sol}.
\end{dcases}
\end{small}
\end{equation}
For \( k=1 \), in our case, we consider the best-known solution as the one with the lowest norm of the gradient of the action functional.
\end{defn}

\noindent The values of \( \rho_{\mathrm{IL}} \) and \( \rho_{\mathrm{TL}_k} \) provide a clear indication of the diversity among local minimum points and the likelihood of transitioning from one level to another. For example, a cluster of minimum points with a large intra-level distance and a small trans-level distance indicates an easy transition to lower objective function values, suggesting a potential underlying funnel structure. For more details, see \cite{Vasile2, funnel_structure}.

\section{Mountain Pass Solutions}
\label{sec:Mountain pass solutions}

The Mountain Pass Theorem addresses the existence of critical points that are not strict local minimisers for a functional \( f \) (for a comprehensive theory on this topic, see \cite{Ambrosetti, Pucci}). An algorithm providing a constructive proof for the Ambrosetti-Rabinowitz Mountain Pass Theorem can be found in \cite{BT04}. In this Section, an improved version of this algorithm is presented.

\noindent Let \( \eta_f : \mathbb{R}_{+} \times X \rightarrow X \) represent the steepest descent flow associated with the functional \( f \), defined as the solution to the Cauchy problem:
\begin{equation}
\label{steepest descent flow}
\begin{cases}
    \dfrac{\text{d}}{\text{d}t} \eta_f (t,x) = - \dfrac{\nabla f(\eta_f (t,x))}{1 + \|\nabla f(\eta_f (t,x))\|}, \\
    \eta_f (0,x) = x.
\end{cases}
\end{equation}
We say a subset \( X_0 \subset X \) is \textit{positively invariant} for the flow \( \eta_f \) if \( \{\eta_f (t,x_0), t \geq 0\} \subset X_0 \) for every \( x_0 \in X_0 \). The \textit{\( \omega \)-limit} of \( x \in X \) for the flow \( \eta_f \) is the closed positively invariant set:
\begin{equation}
\label{omega limit}
\omega_x \coloneq \left\{ \lim_{t_n \to +\infty} \eta_f(t_n, x) : (t_n)_n \subset \mathbb{R}_{+} \right\}.
\end{equation}
Let \( c \in \mathbb{R} \) be such that the sublevel \( f^c \) is disconnected. We denote its disjoint connected components as \( (F_i^c)_i \):
\begin{equation*}
    f^c = \bigcup_i F_i^c, \quad F_i^c \cap F_j^c = \emptyset \quad \text{for} \quad i \neq j. 
\end{equation*}
For each index \( i \), we define the \textit{basin of attraction} of the set \( F_i^c \) as
\begin{equation*}
    \mathcal{F}_i^c \coloneq \{ x \in X : \omega_x \subset F_i^c \}.
\end{equation*}

\begin{defn}[Path]
\label{def: path}
A \textit{path} is a continuous function \( \gamma : [0,1] \rightarrow X \). Given two points \( x_1, x_2 \in X \) with \( x_1 \neq x_2 \), we define the set of paths joining \( x_1 \) to \( x_2 \) as
\begin{equation*}
\label{path}
\Gamma_{x_1, x_2} \coloneq \{ \gamma \in \mathcal{C}([0,1], X) : \gamma(0) = x_1, \gamma(1) = x_2 \}.
\end{equation*}
\end{defn}

\noindent The following can be established:
\begin{thm}[Theorem 1.4, \cite{BT04}]
\label{Thm 1.4}
    Let \( f^c \) be a disconnected sublevel for the functional \( f \). Let \( F_i^c \) be the disjoint connected components of \( f^c \) and \( \mathcal{F}_i^c \) their basins of attraction. For \( x_i \in F_i^c, i = 1,2 \) and \( \gamma \in \Gamma_{x_1, x_2} \), there exists \( \bar{x} \in \gamma([0,1]) \cap \partial \mathcal{F}_1^c \).
\end{thm}

\begin{cor}[Corollary 1.5, \cite{BT04}]
\label{cor 1.5}
Under the same conditions as Theorem \ref{Thm 1.4}, let \( \bar{x} \in \gamma([0,1]) \cap \partial \mathcal{F}_1^c \). Then \( f(\omega_{\bar{x}}) \geq c \) and there exists a sequence \( (x_n)_n = \eta_f(t_n, \bar{x}) \subset \partial \mathcal{F}_1^c \) such that
    \begin{equation*}
        \lim_{n \to + \infty} \nabla f(x_n) = 0 \quad \text{and} \quad \lim_{n \to + \infty} f(x_n) = f(\omega_x).
    \end{equation*}
\end{cor}

\begin{cor}[Corollary 1.6, \cite{BT04}]
\label{cor 1.6}
Under the same conditions as Theorem \ref{Thm 1.4}, let \( \bar{x} \in \gamma([0,1]) \cap \partial \mathcal{F}_1^c \). Then there exists a sequence \( (\tilde{y}_n)_n \subset X \), \( (\tilde{y}_n)_n \coloneq \eta_f(\tilde{T}_n, x_1^n) \) such that
    \begin{equation*}
        \lim_{n \to + \infty} \nabla f(\tilde{y}_n) = 0 \quad \text{and} \quad c \leq f(\tilde{y}_n) \leq f(\bar{x}), \,\, \forall n \in \mathbb{N}.
    \end{equation*}
\end{cor}

\noindent In \cite{BT04}, the proof of Corollary \ref{cor 1.6} is provided through the construction of a sequence \( (\tilde{y}_n)_n \) via an algorithm referred to as Barutello and Terracini's Algorithm. This sequence \( (\tilde{y}_n)_n \) converges when certain additional compactness conditions are imposed on the functional \( f \).

\noindent In this context, we define the following:

\begin{defn}[Palais-Smale Condition]
\label{def: Palais-Smale condition}
    A sequence \( (x_m)_m \subset X \) is called a \textit{Palais-Smale sequence in the interval \( [a, b] \) for the functional \( f \)} if
    \begin{equation*}
        a \leq f(x_m) \leq b, \quad \forall m \in \mathbb{N}, \quad \text{and} \quad \nabla f(x_m) \xrightarrow{m \to + \infty} 0.
    \end{equation*}
    The functional \( f \) satisfies the \textit{Palais-Smale condition in the interval \( [a, b] \)} if every Palais-Smale sequence \( (x_m)_m \) in the interval \( [a, b] \) for \( f \) has a converging subsequence \( (x_m)_k \to x_0 \in X \). 

    \noindent Similarly, a sequence \( (x_m)_m \subset X \) is called a \textit{Palais-Smale sequence at level \( c \) for the functional \( f \)} if
     \begin{equation*}
        f(x_m) \xrightarrow{m \to + \infty} c \quad \text{and} \quad \nabla f(x_m) \xrightarrow{m \to + \infty} 0.
    \end{equation*}
    The functional \( f \) satisfies the \textit{Palais-Smale condition at level \( c \)}, denoted by \( \text{(PS)}_c \), if every Palais-Smale sequence at level \( c \) for \( f \) has a converging subsequence.
\end{defn}

\begin{obs*}
Corollary \ref{cor 1.5} guarantees the existence of a Palais-Smale sequence at level \( f(\omega_{\bar{x}}) \) for the functional \( f \). If \( f \) satisfies \( \text{(PS)}_{f(\omega_{\bar{x}})} \), then there exists a critical point \( \bar{\bar{x}} \) for \( f \) such that \( f(\bar{\bar{x}}) = f(\omega_{\bar{x}}) \). 

\noindent Additionally, Corollary \ref{cor 1.6} implies the existence of a Palais-Smale sequence for \( f \) in the interval \( [c, f(\bar{x})] \). When \( f \) verifies the Palais-Smale condition in \( [c, f(\bar{x})] \), it ensures the convergence of the sequence \( (\tilde{y}_n)_n \) constructed in Barutello and Terracini's Algorithm.
\end{obs*}

\paragraph*{Improved version of Barutello and Terracini's Algorithm}
An improved version of Barutello and Terracini's Algorithm is proposed here. First, we introduce three auxiliary functions: \textsc{Gradient\_Descent} Function, \textsc{Bisection} Function and \textsc{Evolve} Function. In Algorithm \ref{alg: Bisection fucntion}, one can find their definitions considering P as a structure that contains the objective function $f$, its gradient $ \nabla f$, the steepest descent flow function $\eta_f$ as defined in Equation \eqref{steepest descent flow}, the list $L$ of known critical points and $x_{\text{min}}$ as defined in Algorithm \ref{Mountain Pass Algorithm}. Moreover, the parameter $\text{to\_min}$ indicates whether the stopping criteria for the \textsc{Gradient Descent} Function is met either when a minimum point is found or when the gradient $ \nabla f$ starts increasing.

 \begin{algorithm*}[h!]
    \caption{Bisection function}
    \label{alg: Bisection fucntion}
    \begin{algorithmic}[1]

    \Function{Gradient\_descent}{$x, P, \text{to\_min} = \text{true}$}
\State $x_{\text{new}} \gets x, \quad x_{\text{old}} \gets x, \quad dt \gets 1 \times 10^{-3}$
\State $\text{grad}_{\text{new}} \gets \nabla f(x_{\text{new}}), \quad \text{grad}_{\text{old}} \gets \nabla f(x_{\text{old}})$

\While{$(\text{to\_min} \land \| \text{grad}_{\text{new}} \| > \text{g\_tol}) \lor (\neg \text{to\_min} \land \| \text{grad}_{\text{new}} \| \leq \| \text{grad}_{\text{old}} \|)$}
    \State $x_{\text{old}} \gets x_{\text{new}}$
    
    \State $x_{\text{new}} \gets$ integrate $\eta_f$ starting from $x_{\text{old}}$ for one step
    \State $\text{grad}_{\text{old}} \gets  \text{grad}_{\text{new}}$
    \State $\text{grad}_{\text{new}} \gets \nabla f(x_{\text{new}})$
    \If{$ \neg \text{to\_min} \lor (x_{\text{new}} \text{ is known})$}
        \State \Return $x$
    \EndIf
\EndWhile

\State \Return $x_{\text{new}}$
\EndFunction
\\
 \Function{Bisection}{$x_0, x_1, P$}
    \State $\text{distance} \gets \|x_0 - x_1\|$
    \State $x \gets \frac{x_0 + x_1}{2}$  
    \State $x_{\text{new}} \gets$ \Call{Gradient\_descent}{$x, P, \text{to\_min} = \text{true}$}
    \If {$x_{\text{new}}$ is a unknown critical point}
    \State {Add $x_{\text{new}}$ to the list of critical points}
    \State \Return $x_0, x$
    \EndIf
    \If {$\|f(x_{\text{new}}) - f(P.x_{\text{min}})\| < \text{min\_step}$}
        \State \Return $x, x_1$
    \Else
        \State \Return $x_0, x$
    \EndIf
    \EndFunction

    \\

\Function{Evolve}{$x, P$}
 \State \Call{Gradient\_descent}{$x, P, \text{to\_min} = \text{false}$}
\EndFunction
    \end{algorithmic}
    \end{algorithm*}

\noindent We now explain the alternative algorithm to Barutello and Terracini's Algorithm. \textit{Mountain Pass Algorithm} described in Algorithm \ref{Mountain Pass Algorithm} takes initial points \( x_0 \) and \( x_1 \), a function \( f \) and its gradient \( \nabla f \). It returns a non-minimal critical point of the function  \( f \). In details:\\

\noindent \textbf{Input.} Set tolerance values \( \text{g\_tol} \) and \( \text{min\_step} \) to control the stopping criteria for the algorithm and define initial points \( x_0, x_1 \), function \( f \), and its gradient \( \nabla f \). The point \( x_0 \) must be a local minimum, while \( x_1 \) should not lie within the basin of attraction of the steepest descent flow associated with \( \eta_f \) defined in Equation~\ref{steepest descent flow}. Therefore, \( x_1 \) must either be another local minimum or satisfy the condition \( f(x_0) > f(x_1) \).  \\

\noindent \textbf{Step 1.} Initialise the minimum point and set up the problem \( P \) for processing. \\

\noindent \textbf{Step 2.}  Iteratively perform $(x_0, x_1) \gets$ \Call{Bisection}{$x_0, x_1, P$} to approach the stable manifold of the wanted non-minimal critical point. Stop when $\| x_0 - x_1 \| < \text{min\_step}$. 
\begin{obs}
\label{Improvement}
    It is important to highlight that the function \textsc{Bisection} stores any new critical point found during its process. As a matter of fact, when \textsc{Gradient\_Descent} is performed and $x_{\text{new}}$ is found, it checks whether $x_{\text{new}}$ was previously known or not. \\
\end{obs}

\noindent \textbf{Step 3.} Evolve $x_0$ and $x_1$ along the stable manifold until the gradient $\nabla f$ reckoned one of the two points starts increasing. \\

Repeat \textbf{Step 2.} and \textbf{Step 3.} until $\| \nabla f\left(\frac{x_0 + x_1}{2}\right) \| < \text{g\_tol}$. \\

\noindent \textbf{Step 4.} Calculate the midpoint $x$ of \( x_0 \) and \( x_1 \) and return it as the non-minimal critical point. 

\begin{algorithm}[h!]
\caption{Mountain Pass Algorithm}
\label{Mountain Pass Algorithm}
    \begin{algorithmic}[1]
    \State \textbf{Input:} Set the tolerances \text{g\_tol} and min\_step and define the initial points $x_0, x_1$, the function $f$ and its gradient $\nabla f$.
    \State \textbf{Output:} Non-minimal point of the function $f$.
    \State \textbf{Step 1:} Initialise variable: $x_{\text{min}} \gets x_0$; \\ Initialise problem $P$: $P \gets \text{Problem}(f, \nabla f, x_{\text{min}}, L = \{x_0, x_1\})$.
    
        \State \textbf{Steps 2. and 3.} 
        \While{$\|\nabla f\left(\frac{x_0 + x_1}{2}\right)\| > \text{g\_tol}$ }
            \State $(x_0, x_1) \gets \Call{bisection}{x_0, x_1, P}$

            \If{$x_0$ or $x_1$ is known}
                \State \textbf{continue}
            \EndIf
    
            \If{$\|x_1 - x_0\| > \text{min\_step}$}
                \State \textbf{continue}
            \EndIf
            \State $x_0 \gets \Call{evolve}{x_0, P}$
            \State   $x_1 \gets \Call{evolve}{x_1, P}$
        \EndWhile
        
        \State \textbf{Step 4.} Compute and return midpoint: $x \gets \frac{x_0 + x_1}{2}$
        \State \Return $x$
\end{algorithmic}
\end{algorithm}

\noindent Not only does Algorithm \ref{Mountain Pass Algorithm} satisfy the same properties as Barutello and Terracini's Algorithm (allowing it to be used to prove Theorem \ref{Thm 1.4}), but, as pointed out in Remark \ref{Improvement}, it is also designed to identify as many critical points as possible.  \\

\noindent We can now state the Mountain Pass Theorem:
\begin{thm}[Mountain Pass Theorem, Theorem 2.1, \cite{BT04}]
\label{Mountain Pass Theorem}
Let $f$ be a $\mathcal{C}^2$ functional on a Hilbert space $X$. Let $x_1, x_2 \in X$ let $\Gamma_{x_1, x_2}$ be the set of paths defined in \eqref{path} and $c_0$ the level
\begin{equation}
\label{def mountain pass}
    c_0 \eqcolon \inf_{\gamma \in \Gamma_{x_1, x_2}} {\sup_{s \in [0, 1]} {f(\gamma(s))}},
\end{equation}
such that
\begin{equation}
\label{cod mountain pass}
    c_0 > \max{\{f(x_1), f(x_2)\}}.
\end{equation}
If the functional $f$ satisfies the Palais-Smale condition at level $c_0$, then there exists a critical point for the functional $f$ at level $c_0$, that is not a local minimiser.
\end{thm}

\subsection{Mountain pass solutions for the symmetrical $n$-body problem}
\label{sec: Mountain pass solutions for the symmetrical $n$-body problem}
Our aim now is to apply Theorem \ref{Mountain Pass Theorem} to the action functional associated with the $n$-body problem, showing that this theorem implies the existence of a solution distinct from those found in Section \ref{sec:Finding the critical points}. To apply Theorem \ref{Mountain Pass Theorem}, it is necessary to regularise the action functional defined in \eqref{action functional on fundamental domain}.

\noindent

\noindent Among the many approaches that can be used to achieve this regularisation (see, for instance \cite{mountain_pass_strong_force}), we chose the following one, since it more effectively regularises collisional orbits by containing them, rather than pushing them away.
\\
    \begin{equation}
    \label{eq: Regularised action functional}
    \begin{aligned}
        \mathcal{A}_{\I}^{\text{reg.}}(y) &\coloneq  \int_{\I} L_{\varepsilon}(y(t), \dot{y}(t)) \, \mathrm{d}t, \\
        &= \int_{\I} \left(K(\dot{y}(t)) + U_{\varepsilon}(y(t))\right) \, \mathrm{d}t, \\
        &= \int_{\I} \left( \dfrac{1}{2} \sum_{i=1}^n m_i \|\dot{y}(t)\|^2 + \sum_{i<j} \dfrac{m_i m_j}{\sqrt{\varepsilon^2 + \|y_i(t) - y_j(t)\|^2}} \right) \mathrm{d}t.
    \end{aligned}
    \end{equation}
    
    The regularised action functional $\mathcal{A}_{\I}^{\text{reg.}}$ is smooth, specifically $\mathcal{C}^2(Y)$. \\

 We now state the following proposition: \\
\begin{prop}
\label{prop: PS reg action}
    The regularised action functional defined in \eqref{eq: Regularised action functional} satisfies the Palais-Smale condition at every level $c > 0$.
\end{prop}

\begin{proof}
    Let $(y_\nu)_\nu \subset Y$ be such that 
    \begin{equation*}
         \mathcal{A}_{\I}^{\text{reg.}}(y_\nu) \xrightarrow{\nu \to + \infty} c \quad \text{and} \quad \nabla  \mathcal{A}_{\I}^{\text{reg.}} (y_\nu) \xrightarrow{\nu \to + \infty} 0.
    \end{equation*}
    Our aim is to find an element $\bar{y} \in Y$ such that $y_\nu \longrightarrow \bar{y}$ as $\nu \longrightarrow + \infty$ in $Y$. \\
    As coercivity does not depend on the chosen potential $U$ and $\mathcal{A}_{\I}$ is coercive, we have that also $\mathcal{A}_{\I}^{\text{reg.}}$ is coercive. It follows that $(y_\nu)_\nu$ is bounded in $Y$, therefore, up to subsequences, it weakly converges to $\bar{y}$. Moreover, it also follows that we can write $\nabla \mathcal{A}_{\I}^{\text{reg.}}$ as $(\text{Id} + K)$, with $K$ a compact operator; hence:
    \begin{equation}
    \label{Mountain Pass property in the proof}
    \begin{aligned}
         \nabla  \mathcal{A}_{\I}^{\text{reg.}} (y_\nu) &\xrightarrow{\nu \to + \infty} 0 \\
         (\text{Id} + K) (y_\nu) &\xrightarrow{\nu \to + \infty} 0 \\
         y_\nu + K(y_\nu) &\xrightarrow{\nu \to + \infty} 0. 
    \end{aligned}
    \end{equation}
    As we have that $(y_\nu)_\nu \rightharpoonup \bar{y}$ and $K(y_\nu)_\nu$ is precompact, using relation \eqref{Mountain Pass property in the proof}, we have that there exists a subsequence of $(y_\nu)_\nu$, which we still call $(y_\nu)_\nu$, so that $K(y_{\nu})_{\nu} \longrightarrow K(\bar{y})$. Using again relation \eqref{Mountain Pass property in the proof}, we conclude that $y_\nu \longrightarrow \bar{y}$ as $\nu \longrightarrow + \infty$ in $Y$.
\end{proof}

\paragraph{Isolated and non-isolated minimum points}
To apply the Mountain Pass Theorem to the symmetrical $n$-body problem, we must distinguish whether the minimum points found in Section \ref{sec:Finding the critical points} are isolated or not. For this purpose, we introduce the following definition:

\begin{defn}[Type S]
\label{def: type S}
A symmetry group $G$ is said to be of \textit{type S} if for every $x \in Y^G$ there exists a one parameter group of rotation $R_\theta$ such that $R_\theta x \in Y^G$.
\end{defn}

\begin{obs*}
Being of type S implies that if $x$ is a minimum point of the action functional defined in \eqref{action functional on fundamental domain} on $Y^G$, then $R_\theta x$ is also a minimum point. Consequently, this property always implies a geometrically trivial multiplicity of solutions.
\end{obs*}

\noindent This result indicates that, if it holds for a group of rotations parameterised by $R_\theta$ with $\theta \in [a, b]$, then no minimum point can be isolated.

\noindent In our setting, the condition of being \textit{type S} can be formulated as follows: 
\begin{itemize}
    \item[a.] for all elements $g = (\rho, \tau, \sigma) \in \mathrm{ker}\,\tau$, we have $R_\theta^{-1} \rho R_\theta \in \mathrm{ker}\,\tau$;
    \item[b.] if $\tau_g \neq \{\mathds{1}\}$, then $R_\theta^{-1} \rho_g R_\theta = \rho_g$;
\end{itemize}
or equivalently  
\begin{itemize}
    \item[c.] the normaliser $N_{SO(d)}(G)$ of $G$ in $SO(d)$ is a one parameter subgroup of $SO(d)$.
\end{itemize}

\noindent In particular, if we define $V^G = \{x \in \rchi : \rho\, x = x\}$, then $V^G$ must be $R_\theta$-invariant. 

\noindent Accordingly, we can consider the following cases: \\
\begin{itemize}
    \item \textbf{Not of type S:} if the symmetry group $G$ is not of type S, then the minimum points are likely isolated, allowing us to directly apply Algorithm \ref{Mountain Pass Algorithm} to the regularised action functional $\mathcal{A}_{\I}^{\text{reg.}}$ defined in \eqref{eq: Regularised action functional}. \\

    \item \textbf{Type S:} if the symmetry group $G$ is of type S, then the minimum points are  not isolated, and a more intricate analysis would be required to modify Algorithm \ref{Mountain Pass Algorithm}. For example, this would involve distinguishing the minimum points based on their disjoint connected components. However, such an analysis goes beyond the scope and aims of this paper. For the specific case of the examples in Section \ref{section Examples}, one may consider the value of the regularised action functional. \\
\end{itemize}

\noindent In the former case, by Proposition \ref{prop: PS reg action}, we deduce that the regularised action functional $\mathcal{A}_{\I}^{\text{reg.}}$ satisfies the Palais-Smale condition at every level $c > 0$. Therefore, we can apply Theorem \ref{Mountain Pass Theorem}. Using Algorithm \ref{Mountain Pass Algorithm}, we can numerically find a solution that differs from those found in Section \ref{sec:Finding the critical points}.

\section{Examples}
\label{section Examples}

\subsection{Symmetry group $G= \mathbb{Z}_2$} Now, we apply the theoretical framework developed in this article to a concrete example. Considering the three-body problem with equal unitary masses, and the symmetry group \( G = \mathbb{Z}_2 \), we present periodic orbits in \( \mathbb{R}^2 \) identified using Algorithm \ref{alg:MP-AIDEA_Lattice}. These orbits are the circular orbit in Figure \ref{fig: cerchio} of action level $9.802$, the Ducati orbit in Figure \ref{fig: manubrio} of action level $10.442$, the orbit in Figure \ref{fig: es_1 1} of action level $13.579$ and the orbit in Figure \ref{fig: es_1 2} of action level $13.865$.

\noindent For this case, we utilised the code from \cite{symorb} to model the problem, including the cost function used in the optimisation process. Following the setup described in Section \ref{Numerical optimisation routine}, we selected $24$ Fourier coefficients and assumed equal masses for all bodies.

\noindent Regarding the genetic algorithm, we set the number of populations to \( N_\text{P} = 4 \) and the number of agents per population to \( N_\text{A} = 8 \). Each agent was initialized with 24 randomly assigned Fourier coefficients. MP-AIDEA was used to construct the archive \( \textgoth{A} \), which stores all identified local minimum points along with the population individuals at each restart.

\begin{figure}[H]
\centering
\begin{minipage}{0.45\textwidth}
    \centering
    \includegraphics[scale=0.33]{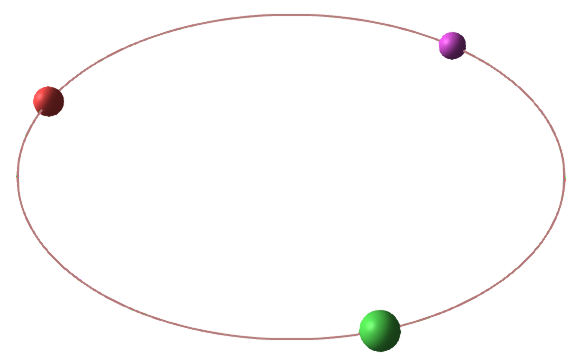}
    \caption{Circular orbit of action level $9.802$}
    \label{fig: cerchio}
\end{minipage}
\hfill
\begin{minipage}{0.45\textwidth}
    \centering
    \includegraphics[scale=0.36]{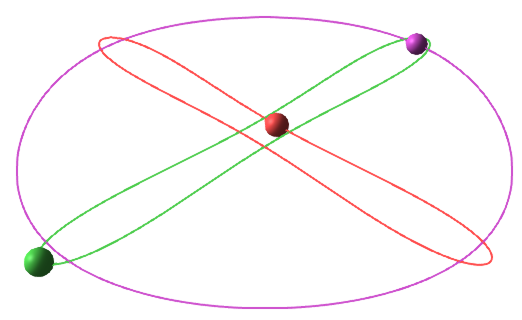}
    \caption{Ducati orbit of action level of $10.442$}
    \label{fig: manubrio}
\end{minipage}
\end{figure}

\begin{figure}[H]
\centering
\begin{minipage}{0.45\textwidth}
    \centering
    \includegraphics[width=0.95\textwidth]{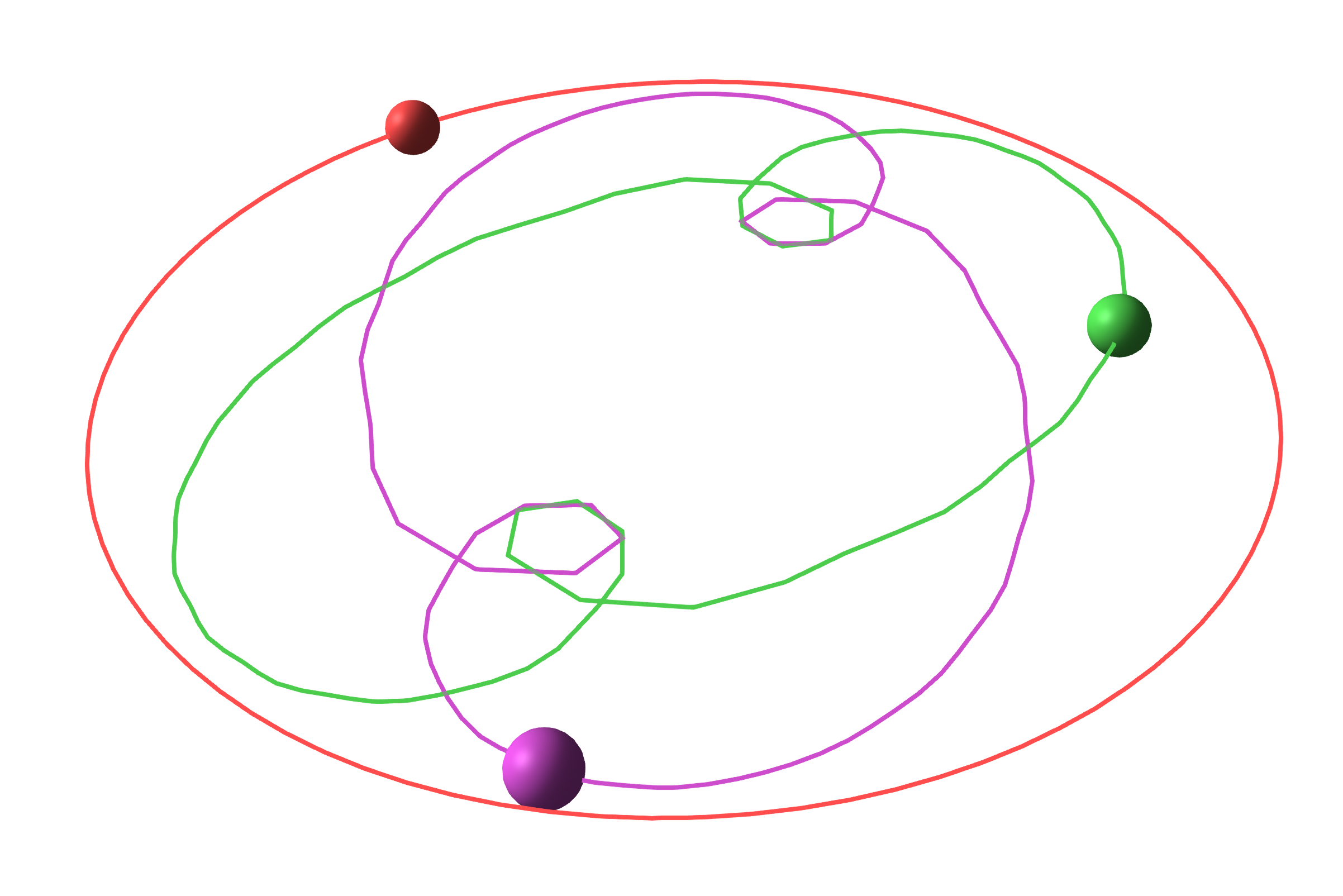}
    \caption{Orbit of action level $13.579$}
    \label{fig: es_1 1}
\end{minipage}
\hfill
\begin{minipage}{0.45\textwidth}
    \centering
    \includegraphics[width=0.95\textwidth]{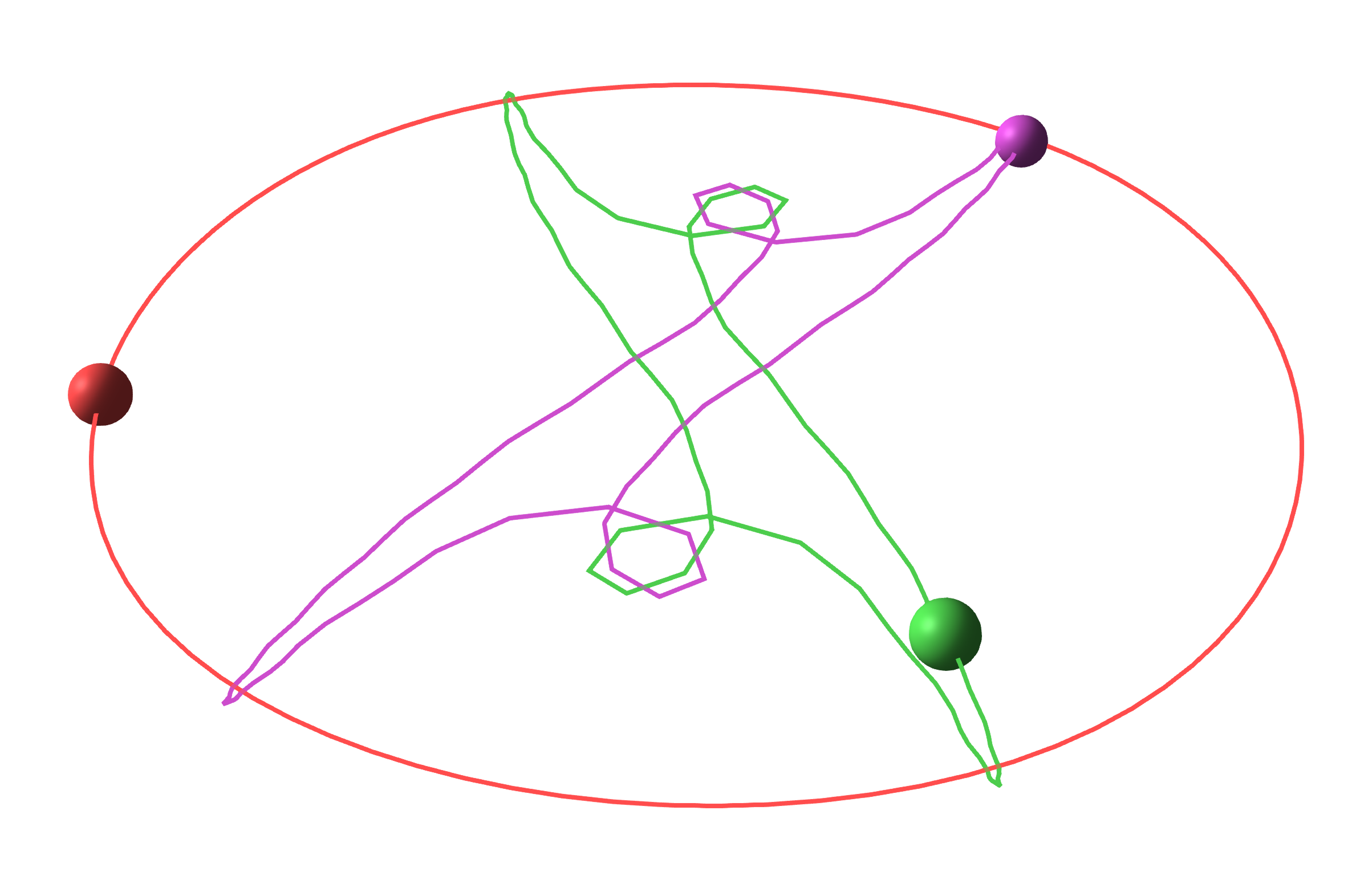}
    \caption{Orbit of action level of $13.865$}
    \label{fig: es_1 2}
\end{minipage}
\end{figure}

\noindent Since this is a relatively simple case, the number of generations in Algorithm \ref{alg:MP-AIDEA_Lattice} did not need to be large. We set \( \text{NoG} = 8 \), which was sufficient to identify four distinct orbits which were the only four solutions found within this setting that exhibited a small gradient norm of the action.

\begin{figure}[H]
\centering
\includegraphics[scale=0.2]{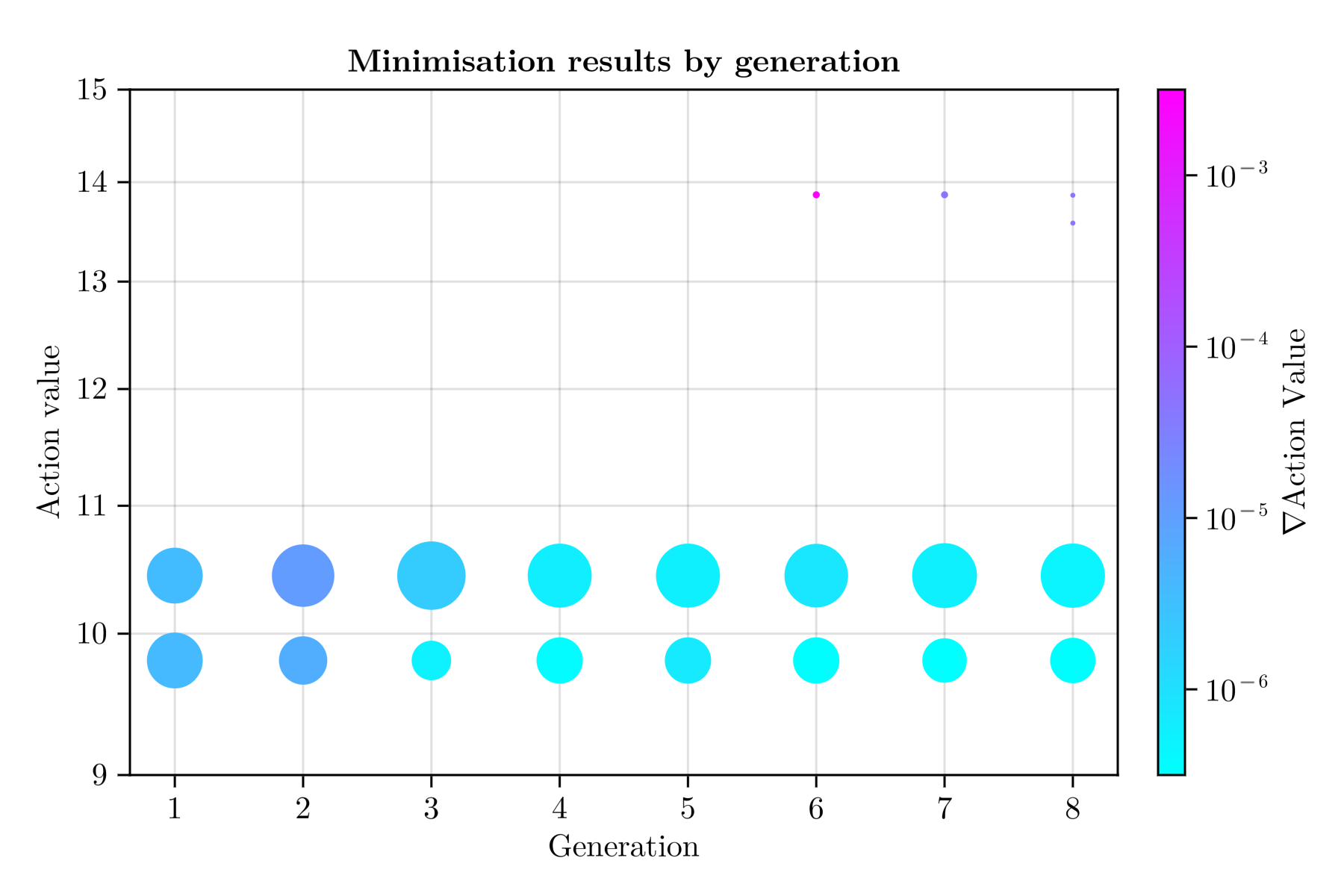}
\caption{Minimisation results by generation. The plot represents the action values of identified solutions across generations. Larger dots indicate a higher number of solutions found at that action value, while the colour gradient represents the magnitude of the action gradient (\(\nabla\) Action Value), with blue indicating smaller gradients and magenta representing higher gradients.}
\label{fig: MPAIDEA_results_1}
\end{figure}

\noindent As shown in Figure \ref{fig: MPAIDEA_results_1}, two primary minimum action values, $9.802$ and $10.442$, were consistently observed across generations. In contrast, the two significant action value points, $13.579$ and $13.865$, only appeared starting from the sixth generation. This result highlights the efficiency of Algorithm \ref{alg:MP-AIDEA_Lattice} in identifying minimum points. As previously mentioned, the  primary objective of Algorithm \ref{alg:MP-AIDEA_Lattice} is to systematically explore the domain to identify as many local minimum points as possible. By leveraging the Lattice algorithm in combination with MP-AIDEA, the search process becomes more structured and adaptive, allowing for a more comprehensive exploration of the solution space.

\paragraph{Mountain Pass solution}
Using the orbits found in Figures \ref{fig: cerchio} and \ref{fig: manubrio} as initial points, the Mountain Pass Algorithm \ref{Mountain Pass Algorithm} gives the following result of action value $11.706$. 
\begin{figure}[H]
\centering
\includegraphics[scale=0.3]{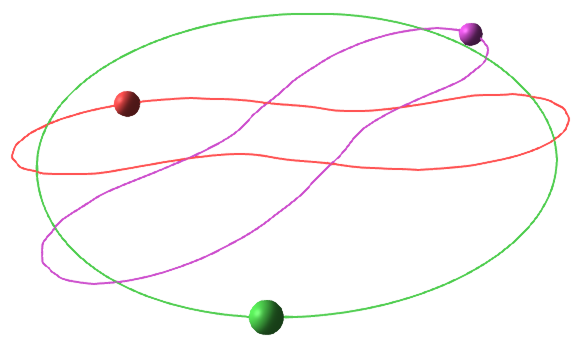}
\caption{Mountain Pass solution}
\label{fig: robaux}
\end{figure}

\paragraph{Morse index}
Figures \ref{fig: cerchio}, \ref{fig: manubrio}, \ref{fig: es_1 1} and \ref{fig: es_1 2} have a discrete Morse index of \( 0 \) on \( \mathbb{I} \), meaning they can be considered local minimum points. On the other hand, Figure \ref{fig: robaux} has a discrete Morse index \( > 0 \) on \( \mathbb{I} \), indicating that it is a non-minimal solution.

\paragraph{Intra and trans level distances}
The values of \( \rho_{\mathrm{IL}} \) and \( \rho_{\mathrm{TL}_k} \) of Figures \ref{fig: cerchio}, \ref{fig: manubrio}, \ref{fig: es_1 1} and \ref{fig: es_1 2} are represented in Figure \ref{fig: intra/trans distances} in blue, yellow, green and pink points respectively. The different points represent the several orbits found using Algorithm \ref{alg:MP-AIDEA_Lattice} with $\text{NoG} = 8$.
\begin{figure}[H]
\centering
\includegraphics[scale=0.17]{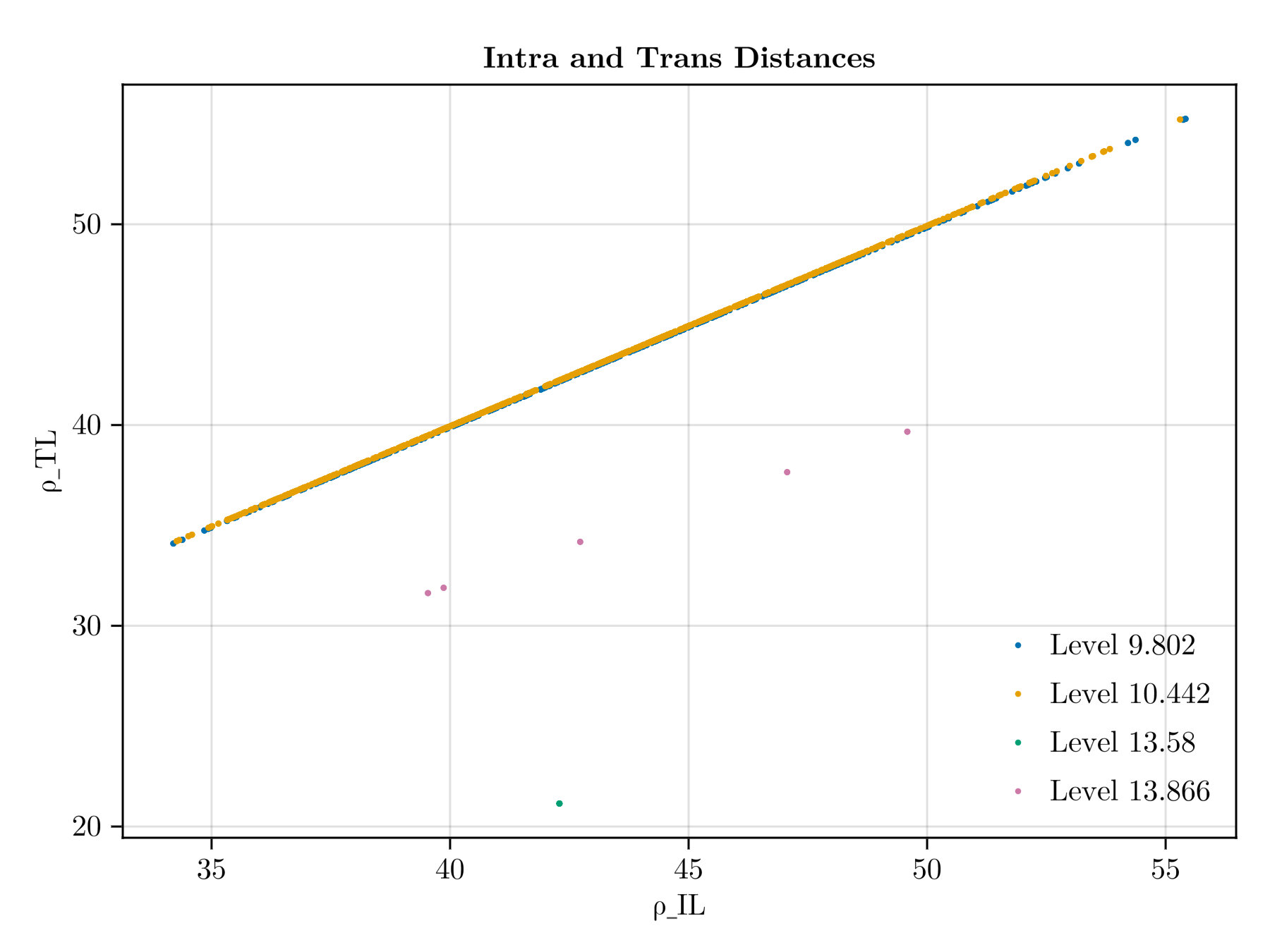}
\caption{The values of \( \rho_{\mathrm{IL}} \) and \( \rho_{\mathrm{TL}_k} \) of Figures \ref{fig: cerchio}, \ref{fig: manubrio}, \ref{fig: es_1 1} and \ref{fig: es_1 2}. Each point represents a local minimum point, with colours distinguishing the two levels.}
\label{fig: intra/trans distances}
\end{figure}

\subsection{Symmetry group $G= D_6$}
Using a similar setting of the previous case, we present another example of a Mountain Pass solution considering a different symmetry group, that is \( G = D_6 \). As initial points for Algorithm \ref{Mountain Pass Algorithm}, we consider two periodic orbits in \( \mathbb{R}^2 \) identified using Algorithm \ref{alg:MP-AIDEA_Lattice} which are the 8-shape orbit of action level of $5.858$ in Figure \ref{fig: 8-shape orbit of action level of $5.858$} and the 8-shape orbit of action level of $9.3$ in Figure \ref{fig: 8-shape orbit of action level of $9.3$}.

\begin{figure}[H]
\centering
\begin{minipage}{0.45\textwidth}
    \centering
    \includegraphics[width=\textwidth]{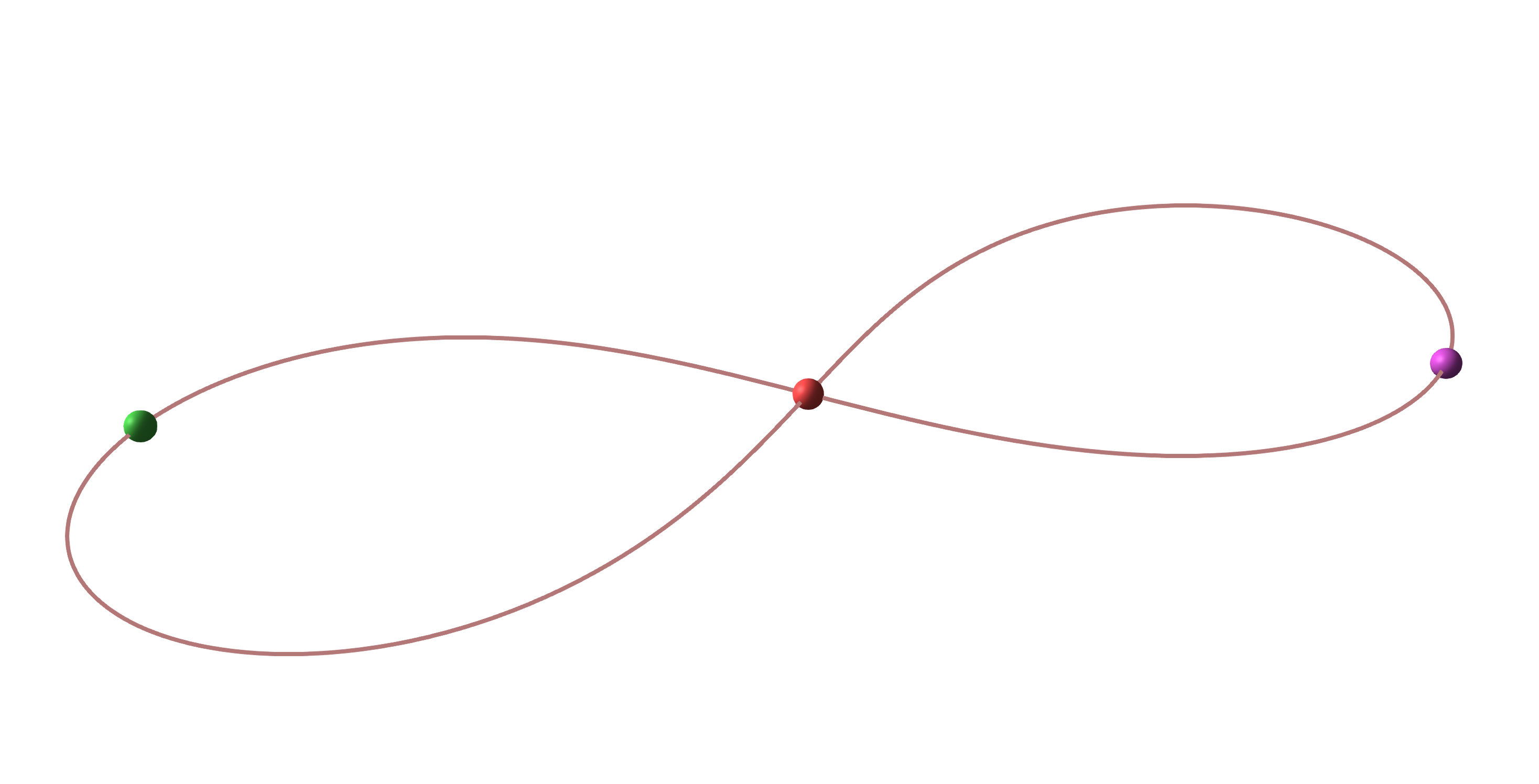}
    \caption{8-shape orbit of action level of $5.858$}
    \label{fig: 8-shape orbit of action level of $5.858$}
\end{minipage}
\hfill
\begin{minipage}{0.45\textwidth}
    \centering
    \includegraphics[width=0.5\textwidth]{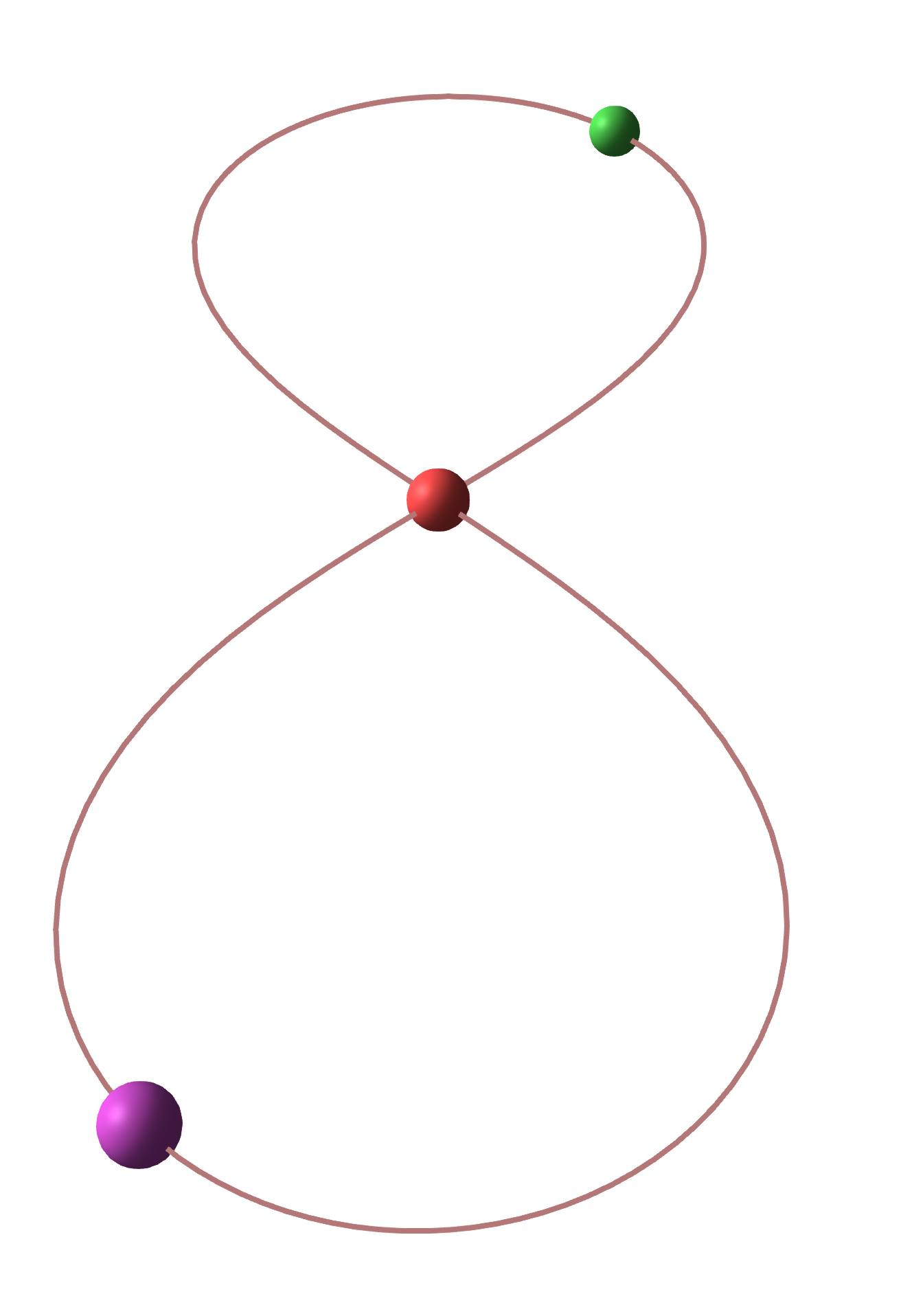}
    \caption{8-shape orbit of action level of $9.298$}
    \label{fig: 8-shape orbit of action level of $9.3$}
\end{minipage}
\end{figure}

\noindent Similar to the previous example, the number of generations in Algorithm \ref{alg:MP-AIDEA_Lattice} did not need to be large. We set \( \text{NoG} = 10 \), which was sufficient to identify two distinct orbits which were the two solutions found within this setting that exhibited a small gradient norm of the action.

\begin{figure}[H]
\centering
\includegraphics[scale=0.2]{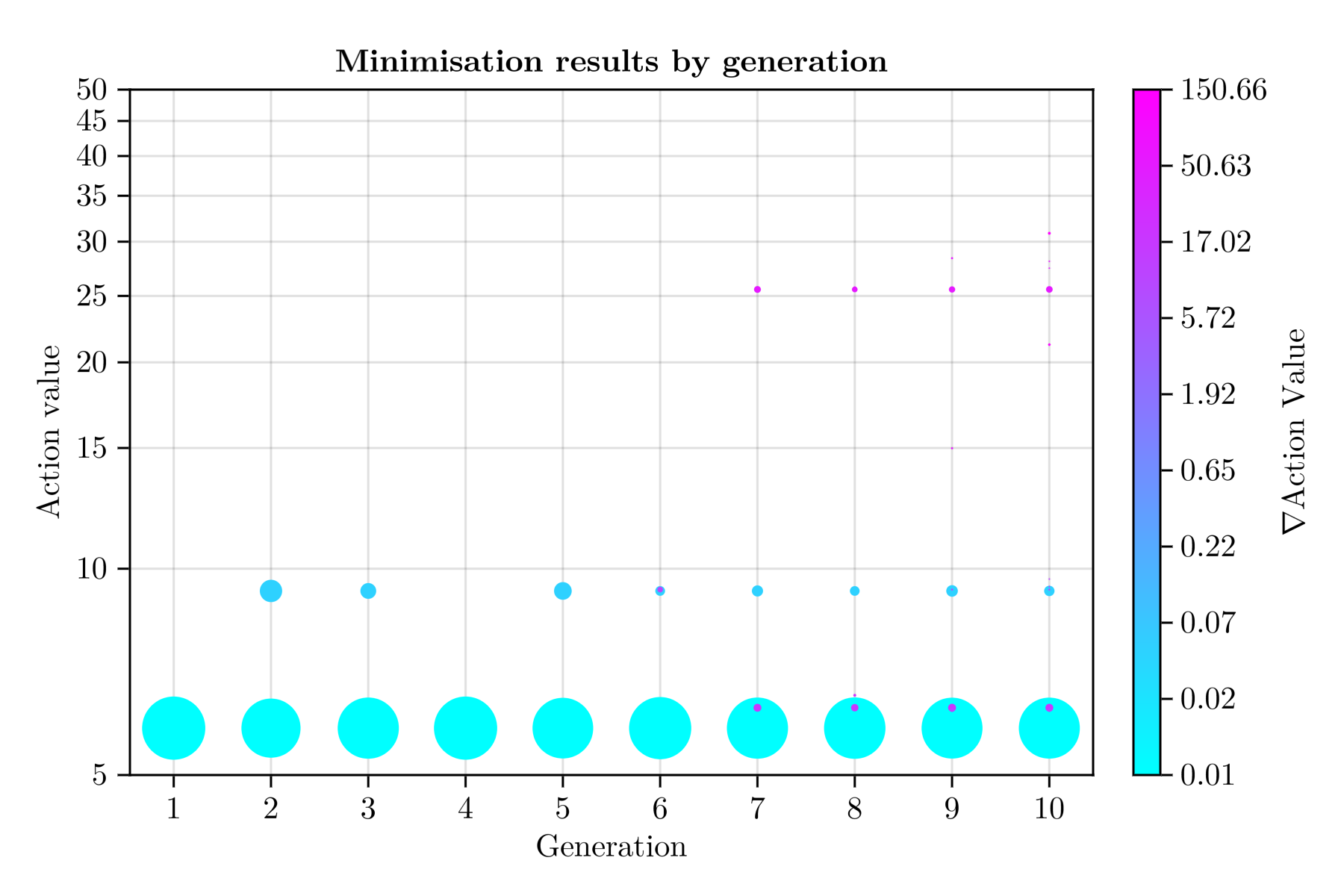}
\caption{Minimization results by generation. The plot represents the action values of identified solutions across generations in logarithmic scale. Larger dots indicate a higher number of solutions found at that action value, while the colour gradient represents the magnitude of the action gradient (\(\nabla\) Action Value), with blue indicating smaller gradients and magenta representing higher gradients.}
\label{fig: MPAIDEA_results}
\end{figure}

\noindent As illustrated in Figure \ref{fig: MPAIDEA_results}, a primary minimum action value of $5.858$ was consistently observed throughout the generations. In contrast, the second significant action value, $9.3$, only emerged from the second generation onward. This outcome aligns with the results of the previous case. Additionally, it is worth noting that from the eighth generation onward, solutions with a high gradient norm also appeared. This is because Algorithm \ref{alg:MP-AIDEA_Lattice} failed to converge in that specific region of the domain. This behaviour is expected, as we are decomposing and minimising on smaller subdomains, some of which may not contain minimisers.

\noindent Using the orbits found in Figures \ref{fig: 8-shape orbit of action level of $5.858$} and \ref{fig: 8-shape orbit of action level of $9.3$} as initial points, the Mountain Pass Algorithm \ref{Mountain Pass Algorithm} gives the following result of action value 9.60:
\begin{figure}[H]
\centering
\includegraphics[width=0.5\textwidth]{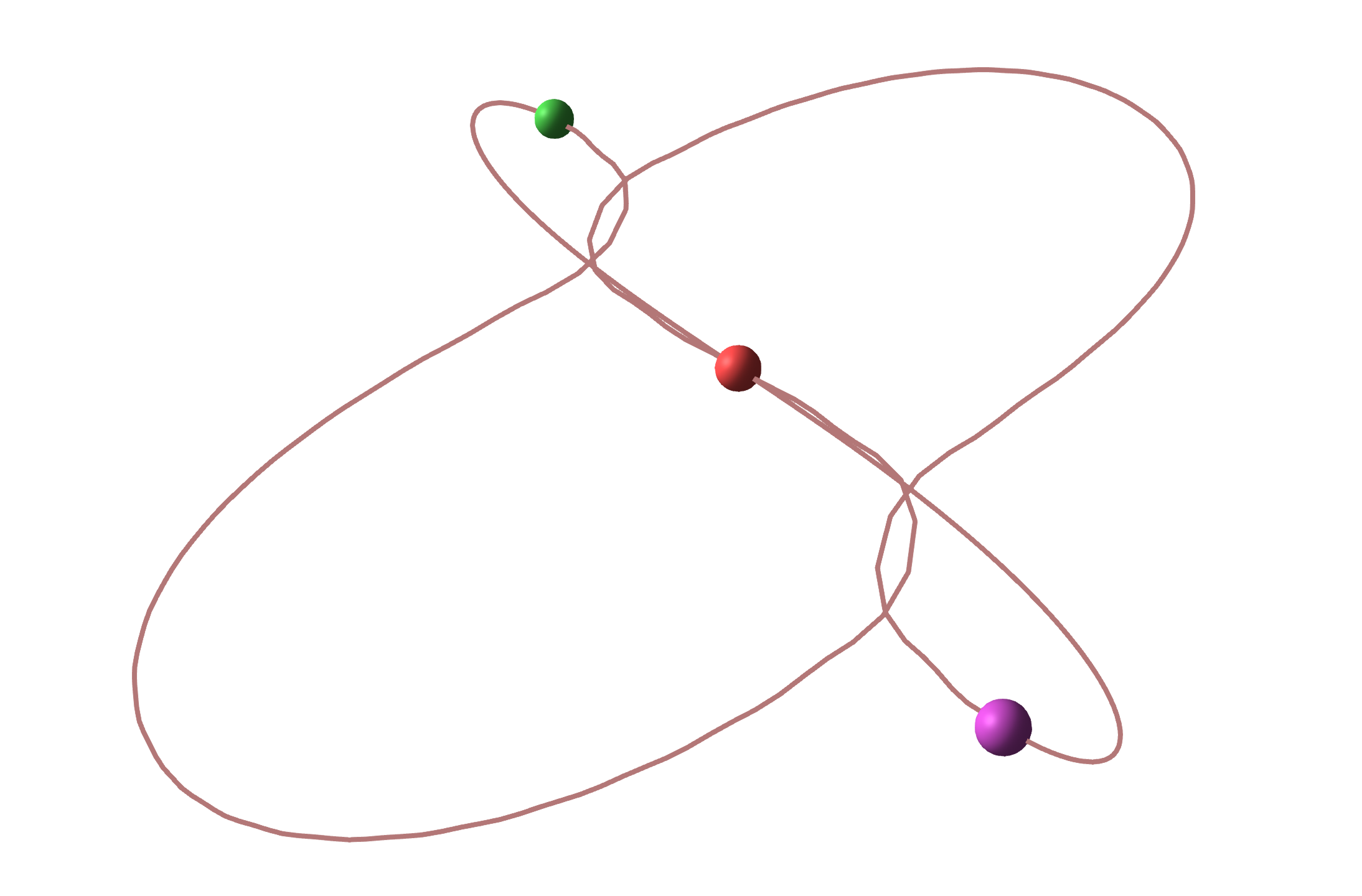}
\caption{Mountain Pass solution}
\label{fig: rolaux}
\end{figure}

\paragraph{Morse index}
Figures \ref{fig: 8-shape orbit of action level of $5.858$} and \ref{fig: 8-shape orbit of action level of $9.3$} have a discrete Morse index of \( 0 \) on \( \mathbb{I} \), meaning they can be considered local minimum points. In contrast, Figure \ref{fig: rolaux} has a discrete Morse index \( > 0 \) on \( \mathbb{I} \), indicating that it is a non-minimal solution.

\paragraph{Intra and trans level distances}
The values of \( \rho_{\mathrm{IL}} \) and \( \rho_{\mathrm{TL}_k} \) of Figures \ref{fig: 8-shape orbit of action level of $5.858$} and \ref{fig: 8-shape orbit of action level of $9.3$} are represented in Figure \ref{fig: intra/trans distances 2} in blue and yellow. The different points represent the several orbits found using Algorithm \ref{alg:MP-AIDEA_Lattice} with $\text{NoG} = 10$.
\begin{figure}[H]
\centering
\includegraphics[scale=0.17]{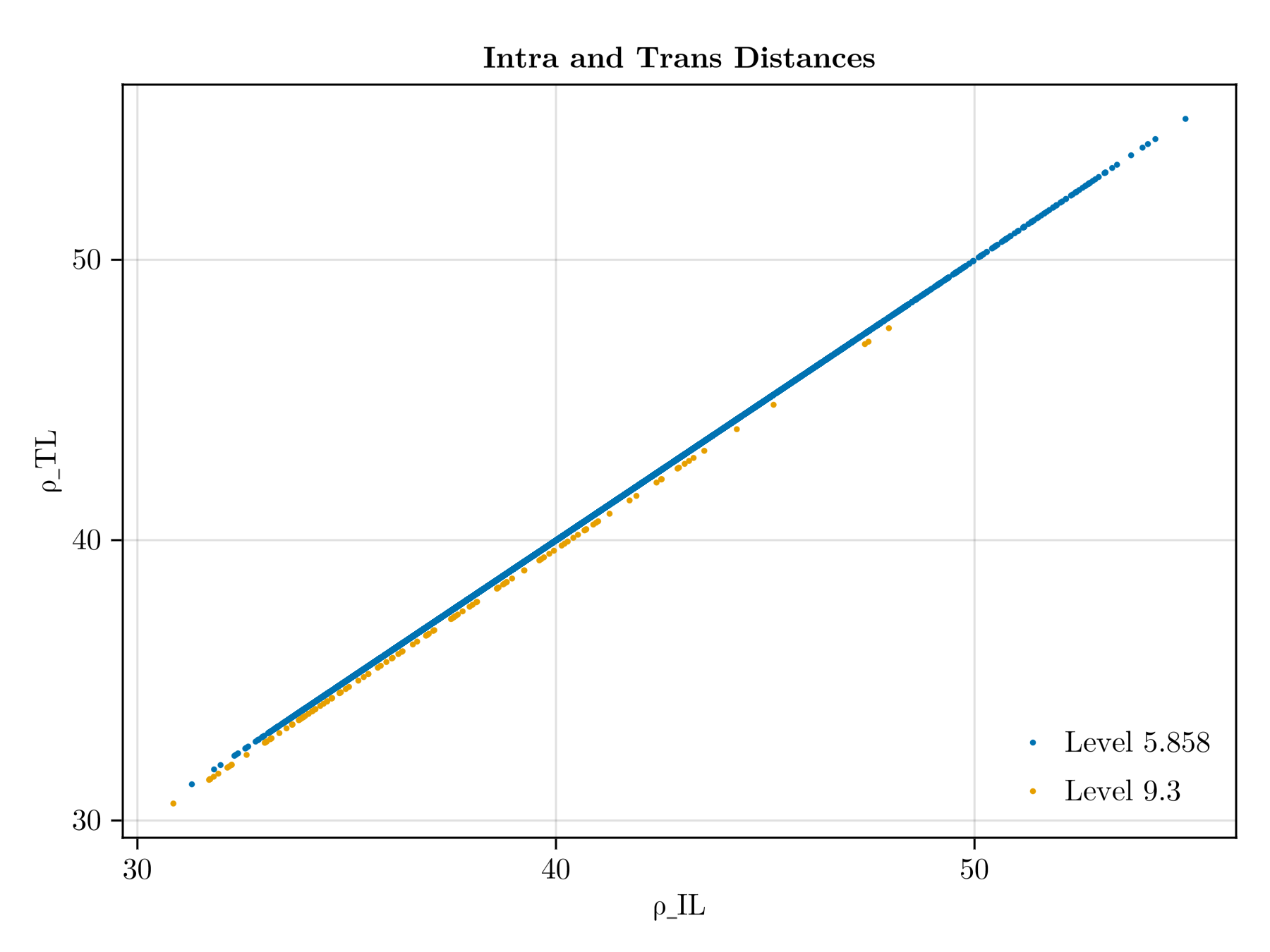}
\caption{The values of \( \rho_{\mathrm{IL}} \) and \( \rho_{\mathrm{TL}_k} \) of Figures \ref{fig: 8-shape orbit of action level of $5.858$} and \ref{fig: 8-shape orbit of action level of $9.3$}. Each point represents a local minimum point, with colours distinguishing the two levels.}
\label{fig: intra/trans distances 2}
\end{figure}

\section{Potential Applications}

While this work does not present an immediate aerospace application, the methodology introduced here has the potential for significant applications. Following Poincaré’s intuition, periodic orbits provide fundamental—if not complete—insight into the underlying chaotic dynamics of a system. Poincaré conjectured that all trajectories in a dynamical system can be approximated by periodic ones \cite{Poincare1892}, stating:  

\begin{quote}
D’ailleurs, ce qui nous rend ces solutions périodiques si précieuses, c’est qu’elles sont, pour ainsi dire, la seule brèche par où nous puissions essayer de pénétrer dans une place jusqu’ici réputée inabordable... \\
\noindent Voici un fait que je n’ai pu démontrer rigoureusement, mais qui me paraît pourtant très vraisemblable. Étant données des équations de la forme définie dans le n. 131 et une solution particulière quelconque de ces équations, on peut toujours trouver une solution périodique (dont la période peut, il est vrai, être très longue), telle que la différence entre les deux solutions soit aussi petite qu’on le veut, pendant un temps aussi long qu’on le veut.
\end{quote}

\begin{quote} \textit{English translation:}  
Moreover, what makes these periodic solutions so valuable is that they are, so to speak, the only opening through which we can attempt to enter a space that was previously considered inaccessible... \\
Here is a fact that I have not been able to demonstrate rigorously, but which still seems very plausible to me. Given equations of the form defined in no. 131 and any particular solution of these equations, one can always find a periodic solution (whose period may, it is true, be very long), such that the difference between the two solutions can be made as small as we wish, for as long a time as we wish.
\end{quote}

\noindent Therefore, the presence of a rich set of periodic solutions serves as a crucial indicator of the complexity of dynamical systems. In an aerospace context, the identification of periodic orbits can be highly relevant for mission design, spacecraft trajectory planning, and the coordination of satellite constellations. By systematically detecting and analysing these orbits, our approach offers a valuable tool for studying dynamical systems with potential applications in celestial mechanics and astrodynamics.

\section{Conclusion}
\label{sec:conclusion}
In this paper, we introduced a more efficient and straightforward approach to identifying additional critical points by combining two algorithms MP-AIDEA, a multi-population adaptive variant of inflationary differential evolution, and Lattice, a domain decomposition method. Next, we turned our attention to the stability of the identified critical points. We  investigated the nature of these critical points calculating their discrete Morse index and visualising the problem's structure using intra- and trans-level distances between local minimum points. Furthermore, we demonstrated the applicability of the Mountain Pass Theorem to a regularised version of the action functional and provided a numerical approach to finding new solutions to the symmetric $n$-body problem, given an initial set of minimum points. While the presented results offer a significant advancement, there remains much more to explore in this area. Further analysis is required to refine the methods, particularly for more complex symmetry groups or higher-dimensional systems, and to extend the numerical techniques for broader classes of problems. These directions hold promise for deeper insights into the structure of critical points and the dynamics of the $n$-body problem.

\section*{Acknowledgement}

This research has been conducted during and with the support of the Italian national inter-university PhD programme in Space Science and Technology. \\
Moreover this work is partially supported by INdAM group G.N.A.M.P.A..
The third author is partially supported by the PRIN 2022 project 20227HX33Z – Pattern formation in nonlinear phenomena

\section*{Declarations}
\textbf{Conflict of interest} The authors have no competing interests to declare that are relevant to the content of this article.

\medskip
\noindent
R. Ciccarelli \orcidlink{0009-0007-8201-536X}\\
Department of Mathematics, University of Turin\\
Via Carlo Alberto 10, 10123 Turin, Italy\\
\texttt{roberto.ciccarelli@unito.it}

\medskip

\noindent
M. Introna \orcidlink{0009-0007-6901-0767}\\
Department of Physics, University of Trento\\
Via Sommarive 14, 38123 Povo, Italy\\
\texttt{margaux.introna@unitn.it}

\medskip

\noindent
S. Terracini \orcidlink{0000-0002-6846-7841}\\
Department of Mathematics, University of Turin\\
Via Carlo Alberto 10, 10123 Turin, Italy\\
\texttt{susanna.terracini@unito.it}

\medskip

\noindent
M. Vasile \orcidlink{0000-0001-8302-6465}\\
Department of Mechanical \& Aerospace Engineering, University of Strathclyde\\
75 Montrose St, Glasgow G1 1XJ, Scotland, United Kingdom\\
\texttt{massimiliano.vasile@strath.ac.uk}


\begin{thebibliography}{10}

\bibitem{DVM20}
Marilena Di Carlo, Massimiliano Vasile, Edmondo Minisci,
\textit{Adaptive multi-population inflationary differential evolution},
Soft Computing, \textbf{24}, 3861–3891, 2020.

\bibitem{BC23}
Vivina Barutello, Gian Marco Canneori,
\textit{Lecture notes in Selected Topics in Celestial Mechanics},
Politecnico of Turin, 2023, March.

\bibitem{CI24}
Isabella Cravero, Margaux Introna,
\textit{A Study on Optimisation Methods in Celestial Mechanics},
Rendiconti del Seminario Matematico Università e Politecnico di Torino, \textbf{82}(2), 2024.

\bibitem{FerrarioTerracini05}
Davide Luigi Ferrario, Susanna Terracini,
\textit{On the existence of collisionless equivariant minimizers for the classical n-body problem},
Inventiones mathematicae, \textbf{155}, 305–362, 2004.

\bibitem{IP24}
Vivina Barutello, Mattia Bergomi, Diego Berti, Gian Marco Canneori, Roberto Ciccarelli, Irene De Blasi, Margaux Introna, Davide Polimeni, Susanna Terracini,
\textit{The factory of equivariant orbits for the 3-body problem},
forthcoming, 2024.

\bibitem{BT04}
Vivina Barutello, Susanna Terracini,
\textit{A bisection algorithm for the numerical Mountain Pass},
Nonlinear Differential Equations and Applications NoDEA, \textbf{14}, 527–539, 2007.

\bibitem{Vasile2}
Massimiliano Vasile, Edmondo Minisci, Marco Locatelli,
\textit{An Inflationary Differential Evolution Algorithm for Space Trajectory Optimization},
IEEE Transactions on Evolutionary Computation, \textbf{15}(2), 267–281, 2011.

\bibitem{Ambrosetti}
Antonio Ambrosetti, Paul H. Rabinowitz,
\textit{Dual variational methods in critical point theory and applications},
Journal of Functional Analysis, \textbf{14}(4), 349–381, 1973.

\bibitem{Pucci}
Patrizia Pucci, James Serrin,
\textit{The structure of the critical set in the Mountain Pass Theorem},
Transactions of the American Mathematical Society, \textbf{299}(1), 115–132, 1987.

\bibitem{Palais}
Richard S. Palais,
\textit{The principle of symmetric criticality},
Communications in Mathematical Physics, \textbf{69}(1), 19–30, 1979.

\bibitem{mountain_pass_strong_force}
Mitsuru Shibayama,
\textit{Minimax approach to the $n$-body problem},
Nonlinear Dynamics in Partial Differential Equations, \textbf{64}, 221–228, 2015.

\bibitem{funnel_structure}
Robert H. Leary,
\textit{Global optimization on funneling landscapes},
Journal of Global Optimization, \textbf{18}, 367–383, 2000.

\bibitem{symorb}
Vivina Barutello, Mattia G. Bergomi, Gian Marco Canneori, Roberto Ciccarelli, Davide L. Ferrario, Susanna Terracini, Pietro Vertechi,
\textit{Equivariant optimisation for the gravitational $n$-body problem: a computational factory of symmetric orbits},
Preprint at \url{https://arxiv.org/abs/2410.17861}, 2024.

\bibitem{nostro}
M. Introna, S. Terracini, M. Vasile, R. Ciccarelli,
\textit{Exploring new orbits for the n-body problem},
75th International Astronautical Conference, Milan, Italy, 14-18 October, Paper number: IAC-24,C1,IP,19,x90420, 2024.

\bibitem{Ambrosetti2}
Antonio Ambrosetti, Vittorio Coti Zelati,
\textit{Periodic solutions of singular Lagrangian systems},
Progress in Nonlinear Differential Equations and their Applications, \textbf{10}, Birkhäuser Boston Inc., 1993.

\bibitem{Arioli}
Gianni Arioli, Filippo Gazzola, Susanna Terracini,
\textit{Minimization properties of Hill's orbits and applications to some N-body problems},
Annales de l'Institut Henri Poincaré C, Analyse non linéaire, \textbf{17}(5), 617–650, 2000.

\bibitem{Bahri}
A. Bahri, P.H. Rabinowitz,
\textit{Periodic solutions of Hamiltonian systems of 3-body type},
Annales de l'Institut Henri Poincaré C, Analyse non linéaire, \textbf{8}(6), 561–649, 1991.

\bibitem{Bessi}
Ugo Bessi, Vittorio Coti Zelati,
\textit{Symmetries and noncollision closed orbits for planar N-body-type problems},
Nonlinear Analysis: Theory, Methods $\&$ Applications, \textbf{16}(6), 587–598, 1991.

\bibitem{Chenciner2}
Alain Chenciner,
\textit{Action minimizing periodic orbits in the Newtonian n-body problem},
Amer. Math. Soc., Providence, RI, Celestial mechanics (Evanston, IL, 1999), 71–90, 2002.

\bibitem{ChencinerMontgomery}
Alain Chenciner, Richard Montgomery,
\textit{A Remarkable Periodic Solution of the Three-Body Problem in the Case of Equal Masses},
Annals of Mathematics, \textbf{152}(3), 881–901, 2000.

\bibitem{Chen}
Kuo-Chang Chen,
\textit{Action-Minimizing Orbits in the Parallelogram Four-Body Problem with Equal Masses},
Annals of Mathematics, \textbf{158}(4), 293–318, 2001.

\bibitem{Ramos}
M. Ramos, S. Terracini,
\textit{Noncollision Periodic Solutions to Some Singular Dynamical Systems with Very Weak Forces},
Journal of Differential Equations, \textbf{118}(1), 121–152, 1995.

\bibitem{ChencinerVenturelli}
Alain Chenciner, Andrea Venturelli,
\textit{Minima de L'intégrale D'action du Problème Newtoniend 4 Corps de Masses Égales Dans R3: Orbites 'Hip-Hop'},
Celestial Mechanics and Dynamical Astronomy, \textbf{77}(2), 139–151, 2000.

\bibitem{Barutello_2004}
Vivina Barutello, Susanna Terracini,
\textit{Action minimizing orbits in the n-body problem with simple choreography constraint},
Nonlinearity, \textbf{17}(6), 2015, 2004.

\bibitem{gronchi}
G. Fusco, G.F. Gronchi, P. Negrini,
\textit{Platonic polyhedra, topological constraints and periodic solutions of the classical N-body problem},
Inventiones mathematicae, \textbf{185}, 283–332, 2011.

\bibitem{MONTALDI_STECKLES_2013}
James Montaldi, Katrina Steckles,
\textit{Classification of symmetry groups for planar $n$ body choreographies},
Forum of Mathematics, Sigma, \textbf{1}, e5, 2013.

\bibitem{Sim}
Carles Simó,
\textit{New Families of Solutions in N-Body Problems},
European Congress of Mathematics, Birkhäuser Basel, 101–115, 2001.

\bibitem{Sim2}
Carles Simó,
\textit{Periodic Orbits of the Planar N-Body Problem with Equal Masses and All Bodies on the Same Path},
IoP Publishing, Bristol, 265–284, 2001.

\bibitem{BarutelloFerrarioTerracini}
Vivina Barutello, Davide Luigi Ferrario, Susanna Terracini,
\textit{Symmetry Groups of the Planar Three-Body Problem and Action-Minimizing Trajectories},
Arch Rational Mech Anal, \textbf{190}, 189–226, 2008.

\bibitem{XIA2022302}
Zhihong Xia, Tingjie Zhou,
\textit{Applying the symmetry groups to study the n body problem},
Journal of Differential Equations, \textbf{310}, 302–326, 2022.

\bibitem{Sim3}
Carles Simó,
\textit{Dynamical properties of the figure eight solution of the three-body problem},
Semanticscholar, 2002.

\bibitem{Nocedal}
Jorge Nocedal, Stephen J. Wright,
\textit{Numerical optimisation},
2nd edition, Springer, New York, NY, 2006.

\bibitem{hillier2005}
F. S. Hillier, G. J. Lieberman,
\textit{Introduction to Operations Research},
8th edition, McGraw Hill, New York, 2005.

\bibitem{Akan2023}
T. Akan, A. M. Anter, A. S. Etaner-Uyar, D. Oliva,
\textit{Engineering Applications of Modern Metaheuristics},
Springer Nature Switzerland AG, 2023.

\bibitem{Gupta2021}
S. Gupta, H. Abderazek, B. S. Yıldız, A. R. Yildiz, S. Mirjalili, S. M. Sait,
\textit{Comparison of Metaheuristic Optimisation Algorithms for Solving Constrained Mechanical Design Optimisation Problems},
Expert Systems with Applications, \textbf{183}, November 2021.

\bibitem{Liu2014}
B. Liu, W. Zhao,
\textit{New Exact Penalty Functions for Nonlinear Constrained Optimization Problems},
Hindawi Publishing Corporation Abstract and Applied Analysis, \textbf{2014}, Article ID 738926.

\bibitem{Meng2005}
Z. Meng, Q. Hu, C. Dang, X. Yang,
\textit{An Objective Penalty Function Method for Nonlinear Programming},
McGrawHill Higher Education, 2005.

\bibitem{OGUNTOLA2021109165}
Micheal B. Oguntola, Rolf J. Lorentzen,
\textit{Ensemble-based constrained optimization using an exterior penalty method},
Journal of Petroleum Science and Engineering, \textbf{207}, 109165, 2021.

\bibitem{Oztas2023}
G. Z. Oztas, S. Erdem,
\textit{A Penalty-based Algorithm Proposal for Engineering Optimisation Problems},
Neural Computing and Applications, \textbf{35}, 7635–765, 2023.

\bibitem{myrdal2013}
Kjartan Kari Gardarsson Myrdal,
\textit{Some Theoretical and Numerical Aspects of the N-Body Problem},
Bachelor's thesis, Lund University Faculty of Science, Centre for Mathematical Sciences, Mathematics, 2013.

\bibitem{Parejo2012}
José Antonio Parejo, Antonio Ruiz-Cortés, Sebastián Lozano, Pablo Fernandez,
\textit{Metaheuristic Optimization Frameworks: A Survey and Benchmarking},
Soft Computing, \textbf{16}(3), 527–561, 2012.

\bibitem{Sorensen2012}
K. Sörensen,
\textit{Metaheuristics—The Metaphor Exposed},
International Transactions in Operational Research, \textbf{22}, 3–18, 2015.

\bibitem{Price2005}
K. Price, R. Storn, J. Lampinen,
\textit{Differential Evolution: A Practical Approach to Global Optimization},
Springer, 2005.

\bibitem{Leary2000}
R. H. Leary,
\textit{Global Optimization on Funneling Landscapes},
Journal of Global Optimization, \textbf{18}(4), 367–383, 2000.

\bibitem{Wales1997}
D. J. Wales and J. P. K. Doye,
\textit{Global Optimization by Basin-Hopping and the Lowest Energy Structures of Lennard-Jones Clusters Containing up to 110 Atoms},
Journal of Physical Chemistry A, \textbf{101}, 5111–5116, 1997.

\bibitem{Kapela2007}
T. Kapela and C. Simó,
\textit{Computer Assisted Proofs for Nonsymmetric Planar Choreographies and for Stability of the Eight},
Nonlinearity, \textbf{20}, 1241–1255, 2007.

\bibitem{Kapela2017}
T. Kapela and C. Simó,
\textit{Rigorous KAM Results Around Arbitrary Periodic Orbits for Hamiltonian Systems},
Nonlinearity, \textbf{30}(3), 965–986, 2017.

\bibitem{Kapela2003}
T. Kapela and P. Zgliczynski,
\textit{The Existence of Simple Choreographies for the N-body Problem—A Computer Assisted Proof},
Nonlinearity, \textbf{16}(6), 1899–1918, 2003.

\bibitem{Moore1993}
C. Moore,
\textit{Braids in Classical Dynamics},
Physical Review Letters, \textbf{70}(24), 3675–3679, 1993.

\bibitem{Moore2006}
C. Moore and M. Nauenberg,
\textit{New Periodic Orbits for the N-body Problem},
Journal of Computational and Nonlinear Dynamics, \textbf{1}(4), 307–311, 2006.

\bibitem{Simo2001b}
C. Simó,
\textit{Periodic Orbits of the Planar N-Body Problem with Equal Masses and All Bodies on the Same Path},
in: Periodic Orbits of the Planar N-Body Problem with Equal Masses and All Bodies on the Same Path, IoP Publishing, Bristol, 265–284, 2001.

\bibitem{Fenucci2022}
M. Fenucci,
\textit{Local Minimality Properties of Circular Motions in \(1/r^\alpha\) Potentials and of the Figure-Eight Solution of the 3-Body Problem},
Partial Differential Equations and Applications, \textbf{3}(1), 10, 2022.

\bibitem{Poincare1892}
H. Poincaré,
\textit{Les méthodes nouvelles de la mécanique céleste},
1892.

\bibitem{turn0search0}
Bert C. Rust and John R. Lewis,
\textit{Finite Difference Discretizations by Differential Quadrature Techniques},
Communications in Numerical Methods in Engineering, \textbf{15}(11), 823–833, 1999.

\bibitem{turn0search3}
Yukun Huang and Adam Oberman,
\textit{Numerical Methods for the Fractional Laplacian: A Finite Difference–Quadrature Approach},
SIAM Journal on Numerical Analysis, \textbf{52}(6), 3056–3084, 2014.

\bibitem{turn0academia19}
Jason E. Hicken and David W. Zingg,
\textit{Summation-By-Parts Operators and High-Order Quadrature},
SIAM Journal on Scientific Computing, \textbf{33}(2), 893–922, 2011.


\end{thebibliography}
\end{document}